\documentclass{lms}

 \usepackage{amsmath,amssymb,xcolor,stmaryrd,nicefrac,mathtools,multicol, changepage, mathabx, cleveref}
\usepackage[inline]{enumitem}

\newenvironment{statement}
  {\vspace{1ex}
   \begin{adjustwidth}{1cm}{1cm}
   \itshape}
  {\end{adjustwidth}
   \vspace{1ex}}

\newcommand{\ignore}[1]{{\color{gray}[Skipped text]}}

\newcommand{\ofun}{{ O}}

\newcommand{\col}{:}

\newcommand{\mco}[1]{#1^*}
\newcommand{\us}{_}

\newcommand{\baseB}[1]{{\rm base}_{ B}(#1)}
\newcommand{\basep}[2]{{\rm base}_{#1}(#2)}

\newcommand{\Chgbases}[2]{ 
\langle \begin{smallmatrix}
  #2\\
  #1
\end{smallmatrix}
\rangle
}

\newcommand{\upgrade}[1]{{\uparrow}_{#1}}
\newcommand{\ug}{{\uparrow}}

\newtheorem{theorem}{Theorem}[section]

\newtheorem{definition}[theorem]{Definition}
\newtheorem{lemma}[theorem]{Lemma}
\newtheorem{remark}[theorem]{Remark}

\newtheorem{example}[theorem]{Example}
\newtheorem{corollary}[theorem]{Corollary}
\newtheorem{proposition}[theorem]{Proposition}

\DeclareFontFamily{U}{mathx}{\hyphenchar\font45}
\DeclareFontShape{U}{mathx}{m}{n}{
      <5> <6> <7> <8> <9> <10>
      <10.95> <12> <14.4> <17.28> <20.74> <24.88>
      mathx10
      }{}
\DeclareSymbolFont{mathx}{U}{mathx}{m}{n}
\DeclareFontSubstitution{U}{mathx}{m}{n}
\DeclareMathAccent{\widecheck}{0}{mathx}{"71}
\DeclareMathAccent{\wideparen}{0}{mathx}{"75}

\newcommand{\B}{\vartheta}

\newcommand{\fs}[2]{ #1 [ #2]}
\newcommand{\fsc}[2]{#1 \{   #2\}}

\newcommand{\fsi}[2]{ {#1}\llbracket{#2}\rrbracket }

\newcommand{\goodp}[3]{{\G}^{#3}_{#2}(#1)}
\newcommand{\G}{\mathbb G}

\newlist{Cases}{enumerate}{9}
\setlist[Cases,1]{label={\sc Case {\rm \arabic*}},wide, labelwidth=!, labelindent=0pt}
\setlist[Cases,2]{label*= {\rm .\arabic*},wide, labelwidth=!, labelindent=0pt}
\setlist[Cases,3]{label*= {\rm .\arabic*},wide, labelwidth=!, labelindent=0pt}
\setlist[Cases,4]{label*= {\rm .\arabic*},wide, labelwidth=!, labelindent=0pt}
\setlist[Cases,5]{label*= {\rm .\arabic*},wide, labelwidth=!, labelindent=0pt}
\setlist[Cases,6]{label*= {\rm .\arabic*},wide, labelwidth=!, labelindent=0pt}
\setlist[Cases,7]{label*= {\rm .\arabic*},wide, labelwidth=!, labelindent=0pt}
\setlist[Cases,8]{label*= {\rm .\arabic*},wide, labelwidth=!, labelindent=0pt}
\setlist[Cases,9]{label*= {\rm .\arabic*},wide, labelwidth=!, labelindent=0pt}

\newcommand{\begincases}{\begin{enumerate}[label*={\sc Case \arabic*},wide, labelwidth=!, labelindent=0pt]}

\newcommand{\begincasesast}{\begin{enumerate*}[label*={\sc Case \arabic*},labelwidth=!, labelindent=0pt]}
\newcommand{\beginclaims}{\begin{enumerate}[label*={\sc Claim },wide, labelwidth=!, labelindent=0pt]}
\newcommand{\begincasesa}{\begin{enumerate}[label={\sc Case ({\rm \roman*})},wide, labelwidth=!, labelindent=0pt]}
\newcommand{\beginsubcases}{\begin{enumerate}[label*= {\rm .\arabic*},wide, labelwidth=!, labelindent=0pt]}
\newcommand{\beginsubcasesast}{\begin{enumerate*}[label*= {\rm .\arabic*},wide, labelwidth=!, labelindent=0pt]}

\newcommand{\putaway}[1]{}

\newcommand{\ve}{\varepsilon}

\renewcommand{\phi}{{\overline{\varphi}}}

\newcommand{\Om}{{\Omega}}

\newcommand{\al}{{\alpha}}

\newcommand{\be}{{\beta}}

\newcommand{\ga}{\gamma}

\newcommand{\de}{\delta}

\newcommand{\N}{\mathbb N}

\newcommand{\om}{\omega}

\newcommand{\FIX}{\mathrm{Fix}}
\newcommand{\JUMP}{\mathrm{Jump}}

\newcommand{\su}{{\mathrm{Succ}}}
\newcommand{\li}{\mathrm{Lim}}

\begin{document}

\title{The Ouroboros Goodstein Principle}

\author[D. Fern\'andez-Duque\textsuperscript{1,2}, M. Morreel\textsuperscript{2} and A. Weiermann\textsuperscript{2}]{David Fern\'andez-Duque\textsuperscript{1,2}, Milan Morreel\textsuperscript{2} and Andreas Weiermann\textsuperscript{2}}

\classno{03F40 (primary), 03D20}

\maketitle

\begin{abstract}
In \cite{fernandez2025fractal}, a variant of Goodstein's original process was recently introduced which, given a set $B\subseteq \N$ of bases, writes each $n\in\N$ in $B$-normal form, namely $n=b^ea+r$, where $b\in B$ the greatest base below $n$. The numbers $e$ and $r$ are then recursively written in $B$-normal form, and finally each base of $B$ is replaced by a corresponding base of some other set $C\subseteq \N$.

The resulting process was shown to terminate and to be independent of $\sf KP$, but the proofs relied on two different ordinal assignments: one monotone but not tight enough to establish independence, and another suitable for independence but not monotone and thus ineffective for proving termination.

We introduce a new ordinal assignment that simultaneously yields termination and independence, thereby revealing the `true' ordinals associated with the numbers in the process.
This assignment allows us to investigate which restrictions to impose on the process in order for the proof-theoretic strength of its termination to lie between the systems $\mathsf{RCA}_0$, $\mathsf{ACA}_0$, $\mathsf{ATR}_0$ and $\mathsf{KP}$. 
\end{abstract}

\section{Introduction}

Goodstein sequences form a classical bridge between elementary number theory and ordinal analysis. The classical Goodstein process goes as follows. Starting from a natural number $n$, we write $n$ in base $b\geq 2$, and we repeatedly write all of the coefficients appearing in this representation in base $b$. In the end we replace all occurrences of $b$ by $b+1$, and we subtract one from the resulting number. We repeat this procedure, increasing $b$ by one at every step.
Goodstein’s principle~\cite{Goodsteinb} then says that every such sequence eventually reaches $0$. The original argument assigns an ordinal below $\ve_0$ to each natural number, and uses the well-foundedness of $\ve_0$ \cite{Goodstein1944}. As $\ve_0$ is the proof-theoretic ordinal of Peano arithmetic ($\sf PA$), this argument can not be carried out in $\sf PA$. In fact, the termination of Goodstein sequences is known to be unprovable in $\sf PA$~\cite{Kirby}.

In the present paper we build upon the results of \cite{fernandez2025fractal}, where the Goodstein principle is extended to a setting in which multiple bases are considered. Whereas that paper emphasizes accessibility, the present work is intended for readers with a stronger background on ordinal analysis.

The general idea of \cite{fernandez2025fractal} is that for each natural number we can consider its $b$-representation, and subsequently consider the representation of the coefficients that appear, but now in some base which is possibly smaller than $b$. We then repeat this process for the new coefficients that appear. Finally, we can examine what happens when we change every base in this expression to a larger one. We call this the upgrade of the natural number we are considering (the definition is made precise in \Cref{section_2}). It is shown in \cite{fernandez2025fractal} that the resulting goodstein principle terminates and is independent of Kripke-Platek set theory ($\sf KP$).

Kripke-Platek set theory ($\sf KP$)~\cite{Barwise} is the fragment of $\sf ZFC$ obtained by omitting the powerset axiom and restricting the comprehension and replacement schemes. In particular separation is limited to $\Delta_0$ formulas and replacement is replaced by the weaker $\Delta_0$-collection scheme.
The proof-theoretic ordinal of $\sf KP$ coincides with that of several well-studied systems, including the theory ${\sf ID}_1$ of non-iterated inductive definitions, and ${\Pi}^1_1$-${\sf CA}^-_0$, the subsystem of second-order arithmetic with parameter-free ${\Pi}^1_1$ comprehension. This common ordinal is denoted $\vartheta[\ve_{\Om+1}]$.

While other variants of Goodstein processes \cite{AraiWW,FernandezFastWalks} typically rely on other functions not provably total in $\sf PA$ in order to establish independence from $\sf KP$, the approach of \cite{fernandez2025fractal} instead considers a very natural principle that remains close to the classical Goodstein process, while achieving substantially higher proof-theoretic strength. Other variants of Goostein's original principle can be found in \cite{FSGoodstein,FernandezWCiE,FernandezWWalk}.

In \cite{fernandez2025fractal}, termination and independence were established using two different ordinal assignments. The assignment used to prove termination was monotone, but not sufficiently tight to yield independence results. Conversely, the assignment used for independence was not monotone, hence failed to establish termination. We present a single ordinal assignment which fulfills both roles. For the independence part we use the fundamental sequences of Buchholz \cite{BuchholzOrd} which are based on Rathjen's $\vartheta$ function \cite{RathjenFragments}. Using this refined assignment we are able to get precise ordinal bounds in \Cref{section_phtr}, which lead to phase transition results for the systems $\mathsf{RCA}_0 + (\Sigma^{0}_{2n}) - \mathrm{IND}$, $\mathsf{ACA}_0$, $\mathsf{ATR}_0$ and $\mathsf {KP^-\om} + (\Pi_n) - \mathrm{IND}$.

\section{Multi-base Goodstein processes}\label{section_2}

We briefly review what was introduced in \cite{fernandez2025fractal} and compare this with the classic Goodstein process.
If $b\in\N\setminus \{0,1\}$, we can write each number $n>0$ in a unique way as $n=b^ea+r$, where $b^e\leq n < b^{e+1}$, $0<a<b$ and $r<b^e$. We call this the {\em $b$-decomposition} of $n$, and we write $n=_b b^ea+r$.

The classic Goodstein process then proceeds as follows. We start with some number $n_0\in\N$. Assume we are at step $i$ of the process and we have obtained the number $n_i$. If $n_i = 0$, then we define $n_{i+1} = 0$. Otherwise, let $b=i+2$, and write $n_i =_b b^ea+r$. Now we write $e$ and $r$ in $b$-decomposition to get a new expression, and we proceed recursively until every number appearing in our expression is at most $b$. Finally, to get $n_{i+1}$, we replace every occurrence of $b$ in our expression by $b+1$ and subtract one. Goodstein's theorem states that, whatever $n_0\in\N$ we start with, $n_i=0$ when $i$ is big enough.

In \cite{fernandez2025fractal}, the notion of $b$-decomposition was extended in the following way. We take a (non-empty) set $B\subseteq\mathbb N \setminus \{0,1\}$ of bases. Given $n$, we define the {\em $B$-decomposition} of $n$ as its $b$-decomposition, where $b\in B$ is maximal satisfying $b\leq n$. If no such $b$ exists, then $b=\min B$.
Formally, $b=\baseB n$, where
\[\baseB n \coloneqq \begin{cases}
    \max\{b\in B \mid b\leq n\} & \text{ if } n\geq\min B. \\
    \min B & \text{ if } n < \min B.    
\end{cases}\]
We then write $n=_B b^ea+r$. 

The new process proceeds in the following way. At each step $i$ of the process we are given a certain set $B_i$ of bases. To obtain $n_{i+1}$ from $n_i$, we write $n_i$ in $B_i$-decomposition, say $n_i=b^ea+r$. Then we write $e$ and $r$ in $B_i$ decomposition, and we proceed recursively until we reach an expression in which every number is either some $b\in B_i$, or less than $\min B_i$. Finally, we replace every base $b\in B_i$ occurring in this expression with a corresponding $c\in B_{i+1}$, and we subtract one. The choice of $c$ is made precise in Definition~\ref{defUpgrade}.

Lastly, for every $n\in\mathbb N$ we define 
\[ S_B(n) \coloneqq
\begin{cases}
    \min\{b\in B\mid b > n \} & \text{ if such a $b$ exists.} \\
    \infty & \text{ otherwise.}
\end{cases}
\]

We further regard every positive integer to be less than, and divide $\infty$.

\begin{definition}
A set $B\subseteq \mathbb N\setminus \{0,1\}$ is called a {\em base hierarchy} if $B\neq\varnothing$ and $b\mid S_B(b)$ for every $b\in B$.
\end{definition}

Note that every singleton $\{b\}$ with $b\geq 2$ is a base hierarchy.

\begin{definition}\label{defUpgrade}
Let $B,C$ be base hierarchies with $\min B\leq \min C$ and $n\in\mathbb N$.
We define $\ug n = \upgrade B^C n\in \mathbb N \cup \{\infty\}$, the upgrade of $n$, recursively on $n$.
If $n<\min B$, then $\ug n = n$.

Otherwise, $n\geq \min B$.
Let $b =\baseB n$ and assume inductively that $\ug m$ is defined for all $m<n$.
We first define an operator $\Chgbases bc = \Chgbases bc_B^C$ on $m\leq n$:
if $m < b$ then $\Chgbases bc m = \ug m $, otherwise write $m=_b b^ea+r$ and set
\begin{equation*}
\Chgbases bc m = c^{\Chgbases bc e } \ug a +\Chgbases bc r .
\end{equation*}
This has the effect of applying the upgrade operator to all the $b$-coefficients of $m$ while changing the base to $d$.

Then, define $\ug n = \Chgbases bc n$, where $c$ is the least element of $C$ such that
\begin{equation}\label{eqdefupgr}
\ug (n-1) < \Chgbases bc n < S_C(c).
\end{equation}
If no such $c$ exists, set $\ug n = \infty$. In the case $c$ does exist, we call it the witness for $\ug n$.

%
%
\end{definition}

In words, given $n\in\mathbb N$, the first step is to perform the usual base change using a base in $C$ that forces the upgrade operator to be monotone at $n$, and then applying the upgrade inductively to all coefficients.
If this yields a  $C$-decomposition, we are done.
Otherwise, search for a suitable base $d$ such that applying the base change to $d$ will yield a $C$-decomposition.
We proceed by giving an example and collecting some results from \cite{fernandez2025fractal}.

\begin{example}\label{exUg}
Suppose that $B =\{ 3,6,42 \}$ and $C=\{ 5,10, 110 \}$.
Let $n = 42^{39} = 42^{6^2 + 3}$; we wish to compute $\ug n = \upgrade  B^C  n$.
We see from the definition that we first have to perform the upgrade of $6^2 + 3$. One easily sees that $\ug 2 = 2$ and $\ug 3 = \Chgbases 35 3 = 5$. We then calculate that $\Chgbases {6} 5 (6^2 + 3) = 5^2 + 5$, which is larger than $S_C(5) = 10$. Thus we instead use the base $10$, and
\[\ug (6^2 + 3) = \Chgbases 6 {10}(6^2 + 3) = 10^2 + 5 = 105 < 110 = S_C(10).\]
Now we see that $\Chgbases {42}{5}n$ and $\Chgbases {42}{10}n$ are both greater than $110$, hence
\[\ug n = \Chgbases {42}{110} (42^{6^2 + 3}) = 110^{105}.\]
\end{example}

The previous example raises some questions. Namely, in calculating $\ug (6^2 + 3)$ and $\ug n$ we have not checked the left inequality of (\ref{eqdefupgr}). This can be justified: by using Lemma~\ref{alternativedefwitness}, one checks that in our previous example, \vspace{5pt}
\begin{enumerate}
    \item $\ug n = \Chgbases 3 5 n$ for $3\leq n < 6$.
    \item $\ug n = \Chgbases 6 {10}n$ for $6\leq n < 42$.
    \item $\ug n = \Chgbases {42}{110}n $ for $42\leq n$.
\end{enumerate}

\begin{lemma}\label{lemmMonUg}
Let $B,C$ be base hierarchies with $\min B\leq \min C$ and $m<n \in\mathbb N$.
Write $\ug$ for $\ug_B^{C}$.

\begin{enumerate}
\item The upgrade operator is monotone, i.e.\ $\ug m < \ug n$.

\item For every $n\geq b\in B$ and $c\in C$ we have $c\leq \Chgbases bc n$.\label{lemmMonUgLowerBnd}

\item The operator $\Chgbases b c$ is monotone if $c\geq \ug b$, i.e.\ $\Chgbases bc m < \Chgbases bc n$ if $c\geq \ug b$.

\item If $c\leq d$, then $\Chgbases bc n \leq \Chgbases bd n$.

\item If $c$ is the witness for $\ug m$ and $d$ is the witness for $\ug n$, then $c\leq d$.\label{lemmMonUgWitnComp}



\end{enumerate}
\end{lemma}

\begin{proof}
The first item can be seen from the definition. For the second item one proves first that for all $b\in B$ and $c\in C$, $\Chgbases bc n > 0$ whenever $n>0$, by induction on $n$. Then the claim follows easily. For the other items, see \cite[Section 3]{fernandez2025fractal}.

\end{proof}

Note, by the first item of the previous lemma, if $\ug n$ has a witness then so does $\ug m$.

\begin{lemma}\label{alternativedefwitness}
    Let $B,C$ be base hierarchies with $\min B\leq \min C$. Let $n\geq \min B$ and $b=\baseB n$. The witness $c$ for $\ug n$ exists iff there is a least element $c'\in C$ such that $c'\geq \ug b$ and $\Chgbases b{c'} n < S_C(c')$. If this is the case, then $c' = c$.
\end{lemma}
\begin{proof}
    By induction on $n$.
    Suppose that $c$ exists. Then $\ug b$ has a witness, say $d$, and $c\geq d = \Chgbases bd b = \ug b$ by the previous lemma. It is clear that $c'\leq c$ exists.

    Suppose that $c'$ exists. If $n=b\in B$, then $c$ exists since $\ug b < \infty$, and it is easily seen that $c'=c$. Otherwise, $n>b$. By the monotonicity of $\Chgbases b {c'}$, we can pick a minimal $d\in C$ satisfying $d\geq \ug b$, $\Chgbases bd (n-1)<S_C(d)$ and $d\leq c'$. By the induction hypothesis, $d$ is the witness for $\ug (n-1)$. It follows that  
    \[\ug (n-1) = \Chgbases bd (n-1) \leq \Chgbases b{c'}(n-1) < \Chgbases b{c'} n < S_C(c').\]
    So $c$ exists, and $c\leq c'$.

    Finally, if one of $c,c'$ exists, then by the previous calculations we have both $c'\leq c$ and $c\leq c'$. So $c'=c$.
\end{proof}

The following lemma says roughly that for every $n\in \N$, we only need a bounded amount of information about $B,C$ to calculate $\ug_B^C n$. This is not surprising given the recursive definition of the upgrade operator.

\begin{lemma}\label{lemmrestrictbasehier}
    Let $B,B',C,C'$ be base hierarchies with $\min B\leq \min C$. Suppose that $n\in\N$ is such that $n\geq \min B$, $B\cap [0,n] = B'\cap [0,n]$ and $C\cap [0,\ug_B^C n]=C'\cap[0, \ug_B^C  n]$. Then, 
    \begin{enumerate}
    \item For $x\leq n$,
    $\ug_B^C x = \ug_{B'}^{C'}x$\label{lemmrestrictbasehier1}
    \item For $x\in\N$, $n\geq b\in B$ and $\ug_B^C n \geq c\in C$, $\Chgbases bc _B^C x = \Chgbases bc _{B'}^{C'} x$.\label{lemmrestrictbasehier2}
    
    \end{enumerate}
\end{lemma}

\begin{proof}
    If $\ug_B^C n = \infty$, the lemma is trivial. So assume otherwise.
    We then prove both items simultaneously by induction on $x$. The case $x<\min B$ is easy. Suppose that $x\geq \min B$.

    For \ref{lemmrestrictbasehier1}, suppose that $x\leq n$. If $x$ is not a base, then the calculation of the upgrade of $x$ uses upgrades as well as applications of the operator $\Chgbases bc$ on numbers less than $x$. Here $b\leq x \leq n$, and $c\leq \ug_B^C x \leq \ug_B^C n$ by the second item of Lemma~\ref{lemmMonUg}. So we can use the induction hypothesis. Otherwise if $x\in B$, then the upgrade of $x$ is the first base which is greater than $\ug (x-1)$. This base exists in $C$ if and only if it exists in $C'$ since $\ug_B^C x \leq \ug_B^C n$. 
    
    For \ref{lemmrestrictbasehier2}, if $x\geq b$, then the calculation of $\Chgbases bc x$ uses upgrades and applications of $\Chgbases bc$ on numbers less than $x$, and we can use the induction hypothesis. Else if $x < b$, then $\Chgbases bc x$ is the upgrade of $x$ and we use \ref{lemmrestrictbasehier1}.
    
\end{proof}

In general, $\ug_B^C(m+1)$ could be much larger than $\ug_B^C m $, but when $m+1$ is a base, we do not want it to be {\em too} much larger.
Good successors ensure that this does not happen.

\begin{definition}
Let $B$ and $C$ be base hierarchies and write $\ug$ for $\ug_B^C$. We say that $C$ is a {\em good successor} of $B$ if the following are satisfied:
\begin{enumerate}[label=\textup{(\arabic*)}]
\item $\min B \leq \min C$.
\item $\ug n < \infty$ for every $n\in\N$, in other words the witness for $\ug n$ always exists.
\item Whenever $\min B< b\in B$, there are no multiples of $  \basep C{\ug (b-1)}$ that lie strictly between $\ug (b-1)$ and $S_C(\ug(b-1))$.
\end{enumerate}
\end{definition}

\begin{example}
In Example~\ref{exUg}, $C$ is a good successor of $B$ because

\begin{enumerate}
    \item $\ug (6-1) = \Chgbases 35 5 = 7$, and there are no multiples of $\basep C 7 = 5$ strictly between $7$ and $S_C(7) = 10$.
    \item $\ug (42 - 1) = \Chgbases 6{10} 41 = 10^2 + 7 = 107$, and there are no multiples of $\basep C {107} = 10$ strictly between $107$ and $S_C(107) = 110$.
\end{enumerate}
\end{example}

\begin{example}\label{exSucc}
Let $B$ be any base hierarchy.
We define $B' = \bigcup B'_n$, where $B'_n$ is defined by induction on $n$ as follows: $B'_n = \{\min B+1\}$ if $n\leq \min B$,  $B'_{n} = B'_{n-1}$ if $\min B < n \notin B$ and otherwise $B'_{n} =B'_{n-1}\cup \{ k\}$, where $k$ is the least multiple of $\max B'_{n-1}$ which is strictly above $\ug_{B}^{B'_{n-1}}(n-1)$. We call $B'$ the minimalistic successor of $B$. It is shown in \cite{fernandez2025fractal} that $B'$ is a good successor of $B$.
\end{example}

Given a base hierarchy $B$, we say an element $n\in\N$ is $B$-critical if $\baseB n\mid n$.
For proofs of the following three lemmas we refer to \cite{fernandez2025fractal}.

\begin{lemma}\label{lemmStructureChgbases}
    Let $B,C$ be base hierarchies with $C$ a good successor of $B$. Let $n\in\N$, $b\in B$ and $c\in C$ such that $c\geq \ug b$.
    \begin{enumerate}
        \item $b\mid n$ if and only if $c\mid \Chgbases bc n$.\label{lemmStructureChgbasesDiv}
        \item If $n=_b b^ea+r$, then $\Chgbases bc n =_c c^{\Chgbases bc e}\cdot \ug a + \Chgbases bc r$.\label{lemmStructureChgbasesRepr}
    \end{enumerate}
\end{lemma}

\begin{lemma}\label{lemmbasecriticalpres}
    Let $B,C$ be base hierarchies with $C$ a good successor of $B$. Let $n\in\N$.
    \begin{enumerate}
        \item $n\in B$ if and only if $\ug n \in C$.
        \item $n$ is $B$-critical if and only if $\ug n$ is $C$-critical.\label{criticalpres}
    \end{enumerate}
\end{lemma}

\begin{lemma}\label{lemmUp(ba+r)}
    Let $B,C$ be base hierarchies with $C$ a good successor of $B$. Let $n\in\N$ and $b\in B$. Write $\ug$ for $\ug_B^C$.
    \begin{enumerate}
        \item If $c$ is the witness for $\ug n$, then $c=\basep C {\ug n}$.\label{lemmUp(ba+r)witnessbase}
        \item If $n=ba+r$ with $r<b$, then $\ug n = \ug ba + \ug r$.\label{lemmUp(ba+r)ba+r}
        \item If $\min B\nmid n$, then $\ug n = \ug (n-1) + 1$.\label{lemmUp(ba+r)minBdiv}
        \item If $b>\min B$, then $\ug b = \ug(b-d) + c$, where $d\in B$ is the predecessor of $b$ in $B$, and where $c= \basep C {\ug (b-1)}$.\label{lemmUp(ba+r)b-d}
    \end{enumerate}
\end{lemma}

\begin{lemma}\label{LemmUpgrTwiceBase}
    Let $B,C$ be base hierarchies with $C$ a good successor of $B$. Let $b,d\in B$ be such that $b=2d$. Write $\ug$ for $\ug_B^C$. Then $\ug b = 2\cdot\ug d$.
\end{lemma}
\begin{proof}
    Let $c = \basep C {\ug(b-1)}\geq \ug d$.
    By the fourth item of the previous lemma, $\ug b = \ug d + c$. Then each of $\ug d$, $c$ and $\ug d +c$ is a base in $C$. In particular $c\mid \ug d +c\leq 2c$, from which $\ug d = c$. Then going back, $\ug b = \ug d + c = 2\cdot \ug d$.
\end{proof}

In the case of good successors, Lemma~\ref{alternativedefwitness} translates to the following result, which yields an alternative definition of the upgrade operator.

\begin{lemma}\label{lemmalternativedefupgr}
    Let $B,C$ be base hierarchies with $C$ a good successor of $B$. Let $n\geq \min B$, $b=\baseB n$, and write $\ug$ for $\ug_B^C$.
    \begin{enumerate}
        \item If $n\in B$, then $\ug n = S_C(\ug(n-1))$.
        \item If $n\notin B$, then the witness for $\ug n$ is the least $c\in C$ for which $c\geq \ug b$ and $\Chgbases bc n < S_C(c)$.
    \end{enumerate}
\end{lemma}
\begin{proof}
    The second item follows immediately from Lemma~\ref{alternativedefwitness} and the fact that $C$ is a good successor of $B$. The first item follows easily from the definition of the upgrade \cite{fernandez2025fractal}.
\end{proof}

As discussed in the beginning of this section, at each step $i$ of the process we consider a base hierarchy $B_i$. Furthermore, for the process to be well defined, we need that $B_{i+1}$ is a good successor of $B_i$ at each step. Such a sequence $(B_i)_{i\in\N}$ is called a {\em dynamical hierarchy.} We will often write $\ug_i$ for $\ug_{B_i}^{B_{i+1}}$, namely the upgrade at step $i$.

\begin{definition}
A {\em dynamical hierarchy} is a sequence $\mathcal B=(B_i)_{i\in\mathbb N}$, where for each $i\in\mathbb N$, $B_i$ is a base hierarchy and $B_{i+1}$ is a good successor of $B_i$.
\end{definition}

\begin{definition}
Given a dynamical hierarchy $\mathcal B$ and $n\in\mathbb N$, we recursively define $\goodp ni{\mathcal B} $ by letting $\goodp n0{\mathcal B} =n $ and if $\goodp ni{\mathcal B}  $ is positive then $ \goodp n{i+1} {\mathcal B} = \upgrade i \goodp n {i} {\mathcal B} -1 $, otherwise $ \goodp n{i+1} {\mathcal B} = 0$.
\end{definition}
 
\begin{example}\label{exClassic}
If $ B_i=\{i+2\}$, then $\mathcal B = ( B_i)_{i\in\mathbb N}$ is a dynamical hierarchy.
The sequence $\goodp ni{\mathcal B} $ is then the classical Goodstein sequence for $n$.
\end{example}

\begin{example}
If $B_0$ is given and we define $B_{i+1} =(B_i)_+$ as in Example~\ref{exSucc}, then $\mathcal B = (B_i)_{i\in\mathbb N} $ is a dynamical hierarchy.
If $B_0=\{2\}$, we obtain the dynamical hierarchy of the previous example.
\end{example}

\section{The Bachmann-Howard Ordinal}\label{secOrd}
 
In this section, we review some notions from ordinal arithmetic and introduce collapsing functions.
We assume basic familiarity with ordinal addition, multiplication, and exponentiation.
The predecessor of $\al$ will be denoted $\al-1$, when it exists.

We will use normal forms for ordinals based on $\Om$, the first uncountable ordinal.
For all ordinals $\xi$, there exist unique ordinals $\al,\be,\ga$ with $\be<\Om$ such that $\xi=\Om^\al\be+\ga$ and $\ga<\Om^\al$.
This is the {\em $\Om$-normal form of $\xi$.}
We define the ordinal $\ve_{\Om+1}$ as the least $\ve>\Om$ such that $\ve=\om^\ve$.
Letting $\Om_0  = 1$ and $\Om_{i+1}  = \Om^{\Om_i }$, we have that $\ve_{\Om+1} = \sup_{n<\om}\Om_i$.
For every ordinal $\xi<\ve_{\Omega+1}$, we define its {\em maximal coefficient} $\mco\xi$ recursively by setting $\mco 0 = 0$ and $\mco{(\Om^\eta\al+\ga)} = \max  \{\al,\mco\eta,\mco\ga\}$. The ordinal assignment we will use is based on Rathjen's $\vartheta$-function \cite{RathjenFragments}, given below.

\begin{definition}
Let $\sup \varnothing =  0$ and define  $\vartheta \colon\ve \us{\Om+1} \to \Om$ by
\[\vartheta(\xi) = \min\{  \om^\theta>\mco \xi \col \forall \zeta<\xi (\mco\zeta <\om^\theta \Rightarrow  \vartheta(\zeta) <\om^\theta) \}.\]
\end{definition}

The function $\vartheta$ provides notations for ordinals below the Bachmann-Howard ordinal, defined as $\vartheta[\ve_{\Om+1}] \coloneqq \sup_{n<\Om}\vartheta(\Om_n)$.
We will often need to compare expressions in terms of $\vartheta$. The following lemma will hence be useful.

\begin{proposition}\cite{BuchholzOrd}\label{propCompareTheta}
If $\zeta < \xi<\ve_{\Om+1}$, then $\vartheta(\zeta)<\vartheta(\xi)$ if and only if $\mco\zeta<\vartheta(\xi)$.
\end{proposition}

Before moving on, we state some additional properties of $\vartheta$.

\begin{lemma}\label{propThetaMonEps}
\begin{enumerate}
\item 
    $\vartheta$ is injective, and surjective on $\vartheta[\ve_{\Om+1}]$.\label{propThetaBij}
\item 
    If $\xi < \ve_{\Om+1}$, then $\mco\xi < \vartheta(\xi)$.\label{propThetaMco}
\item 
    If $\alpha < \ve_{\Om+1}$ and $\beta_1<\beta_2 < \Om$, then
    $\vartheta(\Om\alpha + \beta_1) < \vartheta(\Om\alpha + \beta_2)$.\label{propThetaMon}
    \item If $\Om\leq \xi < \ve_{\Om+1}$, then $\om^{\vartheta(\xi)} = \vartheta(\xi)$.\label{propThetaEps}
    \end{enumerate}
\end{lemma}
\begin{proof}
    \ref{propThetaBij}, \ref{propThetaMco} and \ref{propThetaMon} are shown in \cite{BuchholzOrd}.
    For \ref{propThetaEps} we refer to \cite{FWTheta}.
\end{proof}

\begin{lemma}\label{lemmVarthetaLessOm}
    If $\alpha < \Om$, then
    \[\vartheta(\alpha) = 
    \begin{cases}
        \om^{\alpha + 1}, & \text{if } \alpha = \delta + n \text{ where } \om^{\delta} = \de \text{ and } n<\om.\\
        \om^\al, & \text{otherwise.}
    \end{cases}
    \]
\end{lemma}
\begin{proof}
    Since $\alpha < \Om$, the definition of $\vartheta(\alpha)$ reduces to
    \[\vartheta(\alpha) = \min\{\om^\theta > \alpha : \forall\zeta < \alpha( \vartheta(\zeta) < \om^\theta)\}\]
    Now the proof follows a standard induction on $\alpha$, using the above definition. In the case where $\alpha$ is such that $\om^\al = \al$, we see that we can not take $\theta=\alpha$ but instead we take $\theta = \al + 1$. We leave further details to the reader.
\end{proof}

\section{Termination}

\begin{definition}\label{defO_f^b}
    Let $b\geq 2$ and $f:\N \rightharpoonup \Om$ a partial function with $[0,b)\subseteq \mathrm{dom}(f)$. We define $O\coloneqq O_f^b:\N\rightarrow \ve_{\Om+1}$ as follows. For $n\in \N$, write $n=_b b^ea + r$ and set
    \[O(n) = \begin{cases}
        f(n) & \text{if } n < b,\\ 
        \Om^{O(e)}f(a) + O(r) & \text{if } n\geq b
    \end{cases}\]
\end{definition}

\begin{lemma}\label{lemmO_fMon}
Let $b\geq 2$ and $f:\N\rightharpoonup \Om$ a partial function with $[0,b) \subseteq \mathrm{dom}(f)$. If $f$ is increasing on $[0,b)$, then $O_f^b$ is increasing on $\N$.
\end{lemma}

\begin{proof}[sketch] We briefly indicate the idea. For a detailed proof, see \cite{fernandez2025fractal}. If $m<n$, we can write both $m$ and $n$ in $b$-normal form. Visualizing, 
\[m = b^{e_1}\cdot a_1+\dots + b^{e_k}\cdot a_k < b^{f_1}\cdot c_1 + \dots + b^{f_\ell}\cdot c_\ell = n,\]
where all coefficients lie in $[0,b)$.
Acting with $O^b_f$ on both sides amounts to changing every $b$ into $\Omega$ and applying $f$ to every coefficient. Moreover, $b$-representations are compared lexicographically, and the same holds for $\Om$-representations. Since $f$ is monotone on all of the coefficients we see that $O_f^b$ preserves the order.
\end{proof}

\begin{lemma}\label{lemmO_fMultOm}
    Let $b\geq 2$ and $f:\N\rightharpoonup \Om$ a partial function with $[0,b) \subseteq \mathrm{dom}(f)$. Assume that for $\ve\in\{0,1\}$ we have $f(n)=\ve$ if and only if $n=\ve$. Then $O_f^b(n)$ is a multiple of $\Om$ if and only if $n$ is a multiple of $b$, and $O_f^b(n)$ is a multiple of $\Om^2$ if and only if $n$ is a multiple of $b^2$.
\end{lemma}
\begin{proof}[sketch]
    We indicate the idea and leave further details to the reader. One first proves inductively that for $\ve\in\{0,1\}$, $O_f^b(n) = \ve$ if and only if $n=\ve$. Then, write $n$ in $b$-normal form:
    \[n = b^{e_1}\cdot a_1+\dots + b^{e_k}\cdot a_k.\]
    Applying $O_f^b$, we get
    \[O_f^b(n) = \Om^{O_f^b(e_1)}f(a_1)+\dots + \Om^{O_f^b(e_k)}f(a_k).\]
    Observe that for $\tau\in\{1,2\}$, we have $b^\tau\mid n$ if and only if $e_i> \tau - 1$ for every $i$, if and only if $O_f^b(e_i)>\tau - 1$ for every $i$, if and only if $\Om^2\mid O_f^b$.
\end{proof}

\begin{definition}
Given a base hierarchy $B$, we define two ordinal interpretations, $o_B \colon\mathbb N\to \Omega$ and $\ofun _B \colon \mathbb N\to \ve_{\Om+1}$. If $n<\min B$, we set $o_B(n) = O_B(n)=n$. Otherwise, $n\geq \min B$.

Assume inductively that $o_B$ and $\ofun_B$ are defined on $[0,n)$.
Let $b=\baseB n$, and call $n$ {\em $B$-critical} if $b\mid n$.

We define $O_B(n)$ as $O_f^b(n)$, where $f=o_B\restriction_{[0,b)}$. We assume inductively that for $\ve\in\{0,1\}$, we have $o_B(m) =\ve$ if and only if $m=\ve$ for every $m<n$. Then $f$ satisfies the conditions of Lemma~\ref{lemmO_fMultOm}.

To define $o_B(n)$ we make the following case distinction.

\begin{Cases}

\item ($n = \min B$). Then, $o_B(n) = \vartheta(1) = \om$.

\item ($n=b >\min B$).
Let $d$ be the predecessor of $b$ in $B$ and set
\[o_B(b) = o_B(b-d)\cdot 2.\]

\item ($n$ is not $B$-critical). Write $n=ba+r$ with $0<r<b$ and set
\begin{align*}
 o_B(n) &=  o_B(ba) + o_B(r).
\end{align*}

\item ($n\notin B$ is $B$-critical).
Write $n=b^2u+bv$ with $v<b$.
We will set
\[
o_B(n) = \vartheta(\zeta_B(n)),
\]
where $\zeta_B(n)$ is given as follows. 

By Lemma~\ref{lemmO_fMultOm}, we have that $\Om^2\mid O_B(b^2u)$.
Write $O_B(b^2u) = \Om\cdot\alpha_B(n)$, and let $\beta_B(n)$ be the least ordinal (necessarily less than $\Om$) that satisfies $o_B(b) < \vartheta(\alpha_B(n) + \beta_B(n))$. Then
\[\zeta_B(n) = \alpha_B(n) + \beta_B(n) + \om^{o_B(\tilde v)},\]
where $\om^{o_B(\tilde v)}$ is assumed to be zero if $\tilde v = 0$, and $\tilde v$
given by the following two cases. 
\begin{Cases}
\item ($u=0$ and $2\leq v<\min B$). 
Set $\tilde v = v-2$.
\item ($u>0$ or $\min B\leq v < b$). Set $\tilde v = v$.
\end{Cases}
\end{Cases} 
\end{definition}

Note that in the case where $u=0$, we have that $\alpha_B(n) = 0$ and $o_B(\tilde v) = o_B(v-2) = v-2$, so $\zeta_B(n) = \beta_B(n) + \om^{v-2} < \Om$. In the other case $\zeta_B(n)\geq \Om$.

For all $n\in\N$ and $b\in B$, we define $O_B^b(n)$ as $O_f^b(n)$, where $f=o_B\restriction_{[0,b)}$. We further extend the definitions of $\alpha_B$ and $\beta_B$ to all $n\in \N$ as follows. Set $\alpha_B(0) = \beta_B(0)=0$. If $n>0$, then let $b=\baseB n$ and write $n=b^2u+bv+w$ with $v,w<b$. Let $m=b^2u+bv$. Then, set $\alpha_B(n)=\alpha_B(m)$ and $\beta_B(n) = \beta_B(m)$. It will be convenient to also define $\alpha_B^b(n)$ for all $n\in\N$ and $b\in B$. For this we write $n = b^2u+bv+w$ with $v,w<b$, and let $O_B^b(b^2u) = \Om\cdot\alpha_B^b(n)$. Note $\Om\mid\alpha_B^b(n)$.

We start by collecting some standard properties of the ordinal interpretations.

\begin{lemma}\label{standardPropOo}
Let $B$ be a base hierarchy and $n\in\N$ with $b=\baseB n$.
\begin{enumerate}
    \item $o_B(n) = o_B(n-1)+ 1$ if and only if $\min B\nmid n$.\label{standardPropOo1}
    \item $o_B(n)\in\li$ if and only if $\min B\mid n$.\label{standardPropOo2}
    \item $o_B(n)$ is an additively indecomposable limit if and only if $n=\min B$ or $n\notin B$ is $B$-critical.\label{standardPropOo3}
    \item $\om^{o_B(n)} = o_B(n)$ if and only if $n\notin B$ is $B$-critical and $n\geq b^2$.\label{standardPropOo4}
    \item 
    For all $d\in B$ we have $O_B^d(n) = O_B^d(n-1) + 1$ if and only if $\min B\nmid n$.\label{standardPropOo5}
    \item
    For all $d\in B$ we have
    $O_B^d(n)\in\li$ if and only if $\min B\mid n$.\label{standardPropOo6}
    \item All the coefficients of $O_B(n)$ are of the form $o_B(r)$ for $r<b$. In particular $\mco{O_B(n)} = o_B(r)$ for some $r<b$.\label{standardPropOo7}
    \item All the coefficients of $\alpha_B(n)$ are of the form $o_B(r)$ for $r<b$. In particular $\mco{\alpha_B(n)} = o_B(r)$ for some $r<b$.\label{standardPropOo8}
    \item $o_B(n)<\vartheta[\ve_{\Om+1}]$.\label{standardPropOo9}
\end{enumerate}    
\end{lemma}
\begin{proof}
    We first prove \ref{standardPropOo1} and \ref{standardPropOo2} by induction on $n$. The case $n\leq\min B$ is trivial. 
    Suppose $n > \min B$. If $n\in B$, then $\min B\mid n$ and $o_B(n) = o_B(n-d)\cdot 2$, where $d$ is the predecessor of $n$. Then \ref{standardPropOo2} follows by the induction hypothesis, and \ref{standardPropOo1} follows from \ref{standardPropOo2}.
    If $n$ is not $B$-critical, then write $o_B(n) = o_B(ba) + o_B(r)$, where $0<r<b$. Both \ref{standardPropOo1} and \ref{standardPropOo2} follow by applying the induction hypothesis on $r$.
    Suppose that $n\notin B$ is $B$-critical, and write $n=b^2u+bv$ with $v<b$. If $u=0$, then by the definition of $\beta_B(n)$ and Lemma~\ref{propThetaMonEps}\ref{propThetaMon}, $o_B(n) = \vartheta(\beta_B(n) + \om^{v-2})>o_B(b)$. By the induction hypothesis, $o_B(b)\in\li$, and since the image of $\vartheta$ consists of additively decomposable ordinals, $o_B(n)\in\li$.
    Otherwise if $u>0$, then $\alpha_B(n)>0$ and $o_B(n) = \vartheta(\zeta_B(n))\in\li$ by Lemma~\ref{propThetaMonEps}\ref{propThetaEps}. This proves \ref{standardPropOo1} and \ref{standardPropOo2}.

    Items \ref{standardPropOo3} and \ref{standardPropOo4} follow easily by inspecting the definition and using Lemma~\ref{propThetaMonEps}\ref{propThetaEps}. For \ref{standardPropOo5} and \ref{standardPropOo6}, write $n=d^{e_1}\cdot a_1 + \dots + d^{e_k}\cdot a_k$ in $d$-normal form. Then 
    \begin{equation}\label{eqnormalFormminBnotdiv}
    O_B^d(n) = \Om^{O_B^d(e_1)}o_B(a_1)+\dots + \Om^{O_B^d(e_k)}o_B(a_k). \tag{$\star$}
    \end{equation}
    If $\min B\mid n$, then either $e_k > 0$ or $\min B\mid a_k$. If $e_k>0$, then $O_B^d(e_k)>0$ and $O_B^d(n)$ is a multiple of $\Om$. Otherwise $O_B^d(n)$ ends with $o_B(a)\in\li$ by \ref{standardPropOo2}. In any case, $O_B^d(n)\in\li$. On the other hand if $\min B\nmid n$, then $e_k=0$ and $\min B\nmid a_k$. It is clear that writing out $O_B^d(n-1)$ would give us (\ref{eqnormalFormminBnotdiv}) with $o_B(a_k)$ replaced by $o_B(a_k-1)$. By \ref{standardPropOo1}, $o_B(a_k) = o_B(a_k-1) + 1$, and we conclude that $O_B^d(n) = O_B^d(n-1) + 1$.
    
    For \ref{standardPropOo7}, we have $O_B(n) = O_f^b(n)$ where $f=o_B\restriction_{[0,b)}$ by definition, and the coefficients of $O_f^b(n)$ are of the form $f(r) = o_B(r)$ for $r<b$.

    For \ref{standardPropOo8}, write $n=b^2u+bv+w$. By \ref{standardPropOo7}, the coefficients of $O_B(b^2u)$ are of the form $o_B(r)$ with $r<b$ (this holds trivially if $u=0$). Write $O_B(b^2u) = \Om^{\lambda_1}\mu_1 + \dots + \Om^{\lambda_n}\mu_n$ in $\Om$-normal form.
    Then $\alpha_B(n) = \Om^{\lambda_1'}\mu_1 + \dots + \Om^{\lambda_n'}\mu_n$, where for $1\leq i\leq n$ we have $\lambda_i' = \lambda_i$ if $\lambda_i\geq \om$, and $\lambda_i' = \lambda_i-1$ otherwise. 
    Now note that $\lambda_i\geq 2$, and the coefficients of $\lambda_i$ are also of the form $o_B(r)$ with $r<b$. So, in the case where $\lambda_i < \om$, we have $\lambda_i = o_B(r)$ for some $0<r<b$. By \ref{standardPropOo1} and \ref{standardPropOo2}, $\lambda_i' = o_B(r-1)$.

    Item \ref{standardPropOo9} is easy to prove by induction on $n$, using Lemma~\ref{propCompareTheta} in the case where $n\notin B$ is $B$-critical, together with \ref{standardPropOo8} and the fact that $\beta_B(n)\leq o_B(b)$, since $o_B(b) < \vartheta(\alpha_B(n)+o_B(b))$ by Lemma~\ref{propThetaMonEps}\ref{propThetaMco}.
\end{proof}

Though it will not be needed, one can easily prove the following analogue of Lemma~\ref{lemmrestrictbasehier} by induction.
\begin{lemma}
    Let $B,B'$ be base hierarchies. Suppose that $n\in\N$ is such that $B\cap [0,n] = B'\cap [0,n]$. Then,
    \begin{enumerate}
        \item For $x\leq n$, $o_B(x) = o_{B'}(x)$.
        \item For $x\in\N$ and $b\leq n$, $O_B^b(x) = O_{B'}^b(x)$.
    \end{enumerate}
\end{lemma}

\begin{proposition}\label{propMonO}
If $B$ is a base hierarchy and $m<n$, then 
\begin{enumerate}
    \item $o_B(m)<o_B(n)$.\label{Mon_o}
    \item If $b=\baseB m = \baseB n$, then $O_B(m) < O_B(n)$.\label{Mon_O}
    \item If $b=\baseB m = \baseB n$, then
    $\alpha_B(m) \leq \alpha_B(n)$.\label{Mon_alpha}
\end{enumerate}
\end{proposition}

\begin{proof}
We first show that \ref{Mon_O} implies \ref{Mon_alpha}. Indeed, if $b=\baseB m = \baseB n$, then we can write $m=b^2u'+bv'+bw'$ and $n=b^2u+bv+bw$ with $v,v',w,w'<b$. Then $m<n$ implies $u'\leq u$, from which $O_B(b^2u') \leq O_B(b^2u)$ and $\alpha_B(m)\leq \alpha_B(n)$.

We prove \ref{Mon_o} and \ref{Mon_O} simultaneously by induction on $m+n$. The case $n\leq\min B$ is trivial (note that $O_B(\min B) = \Om$), so we assume otherwise.

By the induction hypothesis, we know that $o_B$ is increasing on $[0, b)$. It then follows from Lemma~\ref{lemmO_fMon} that $O_B(m) < O_B(n)$. For the rest of the proof, we focus on the inductive step for the first item.
Let $b = \baseB n$ and $d=\baseB m$.
\begin{Cases}

\item ($n=b>\min B$).
Write $m=da+r < b$. Since $b$ is a multiple of $d$, $da \leq b-d$ and $r< d \leq b-d$, so the induction hypothesis yields
\[o_B(m) = o_B(da) + o_B(r) < o_B(b-d)\cdot 2 = o_B(n).\]

\item ($b\nmid n$). Write $n=ba+r$ with $0<r<b$. If $m<ba$, then the induction hypothesis gives $o_B(m) < o_B(ba) < o_B(n)$. Otherwise $m=ba+r'$, and we apply the induction hypothesis to $r'<r$.

\item ($b\mid n>b$). 
Write $n=b^2u+bv$ with $v<b$.
If $d < b$ then the induction hypothesis gives $o_B(m) < o_B(b)$ and by the definition of $\beta_B(n)$ we have $o_B(b) < \vartheta(\alpha_B(n) + \beta_B(n))\leq o_B(n)$. So we can assume that $d=b$ and $m>b$. Furthermore, if $m=ba+r$ with $0<r<b$ then by the induction hypothesis $o_B(m) < o_B(ba)\cdot 2$. Since $o_B(n)$ is additively indecomposable by Lemma~\ref{standardPropOo}\ref{standardPropOo3}, it would suffice to prove that $o_B(ba) < o_B(n)$. So we can assume that $m = b^2u' + bv'$ for some $u'$ and $v' < b$.
Then $o_B(m) = \vartheta(\zeta_B(m))$ and $o_B(n) = \vartheta(\zeta_B(n))$.

\begin{Cases}
\item ($u'=u$).
This case follows easily from Lemma~\ref{propThetaMonEps}\ref{propThetaMon} by using the induction hypothesis on $v'<v$ and noting that $u=0$ and $2\leq v<\min B$ if and only if this is also the case for $u',v'$.

\item($u'<u$). We can assume that $n=b^2u$ by the previous case. Regardless of $u'$ and $v'$, we have that $o_B(m) \leq \vartheta(\alpha_B(m) + \beta_B(m) + \om^{o_B(v')})$, and we will show that the right-hand side is less than $o_B(n) = \vartheta(\alpha_B(n) + \beta_B(n))$. 
By (the proof of) item \ref{Mon_alpha} we have that $\alpha_B(m) < \alpha_B(n)$, and since $\alpha_B(m)$ and $\alpha_B(n)$ are multiples of $\Om$,
\[\alpha_B(m) + \beta_B(m) + \om^{o_B(v')} < \alpha_B(n) + \beta_B(n).\]
By Proposition~\ref{propCompareTheta}, we are left with showing that
\[\mco{(\alpha_B(m) + \beta_B(m) + \om^{o_B(v')})} < \vartheta(\alpha_B(n) + \beta_B(n)) = o_B(n).\]
Since every coefficient of $\alpha_B(m)$ is of the form $o_B(r)$ with $r<b$, $\mco{\alpha_B(m)} < o_B(n)$ by the induction hypothesis. From the induction hypothesis we also get that $o_B(v') < o_B(n)$. By Lemma~\ref{propThetaMonEps}\ref{propThetaEps}, $\omega^{o_B(v')} < o_B(n)$, and since $o_B(n)$ is additively indecomposable it remains to show that $\beta_B(m) < o_B(n)$. But this follows from the minimality of $\beta_B(m)$ and because $o_B(b) < \vartheta(\alpha_B(m) + o_B(b)) < \vartheta(\alpha_B(m) + o_B(n))$. Here we have used Lemma~\ref{propThetaMonEps}\ref{propThetaMco} and Lemma~\ref{propThetaMonEps}\ref{propThetaMon}, respectively.
\end{Cases}
\end{Cases}
\end{proof}



The following inductive step will be useful.

\begin{lemma}\label{lemmInductiveStepAlpha}
    Let $\alpha$ be a multiple of $\Om$ and let $\min B < m$. Let $\de<\Om$ be such that for all $m'<m$ holds $o_B(m') < \vartheta(\alpha + \de)$. Suppose that either $\alpha_B(m) <  \alpha$ or $o_B(m)  \leq \mco\alpha$. Then also $o_B(m)<\vartheta(\al+\de)$.
\end{lemma}

\begin{proof}
    If $o_B(m)\leq \mco\alpha$, then the claim follows from Lemma~\ref{propThetaMonEps}\ref{propThetaMco}. So we can assume that $\mco\alpha < o_B(m)$, and $\alpha_B(m) < \alpha$.

    The cases where $m$ is not $B$-critical or $m\in B$ follow from the inductive hypothesis and the fact that $\vartheta(\al + \de)$ is additively indecomposable. So suppose that $m\notin B$ is $B$-critical, and write $m=d^2u + dv$ with $d=\baseB m$. Then we have to show that
    \[o_B(m) = \vartheta(\zeta_B(m)) = \vartheta(\alpha_B(m) + \beta_B(m) + \om^{o_B(\tilde v )}) < \vartheta(\alpha + \de),\]
    where $\tilde v$ is either $v-2$ or $v$. We will use Proposition~\ref{propCompareTheta}. 
    
    From $\alpha_B(m)<\alpha$ we get $\zeta_B(m) < \al+\de$ and $\al>0$.
    Furthermore, we have
    $\mco{\alpha_B(m)}, o_B(\tilde v) < o_B(d) < \vartheta(\alpha + \de)$, and by the minimality of $\beta_B(m)$, also $\beta_B(m) \leq o_B(d) < \vartheta(\alpha + \de)$. By Lemma~\ref{propThetaMonEps}\ref{propThetaEps}, $\beta_B(m) + \om^{o_B(\tilde v)} < \vartheta(\alpha + \de)$.
\end{proof}

Given $n$ which is $B$-critical with $b = \baseB n$, we will let $n_*$ denote the greatest $B$-critical element not in $B$, which is less than $b$ and satisfies $\alpha_B(n)\leq \alpha_B(n_*)$ and $(\alpha_B(n))^* <  o_B(n_*)$, if it exists. By a candidate for $n_*$ we shall mean a $B$-critical element $m\notin B$ such that $m<b$, $\alpha_B(n)\leq \alpha_B(m)$ and $(\alpha_B(n))^* < o_B(m)$. By the previous lemma we obtain the following.

\begin{lemma}\label{lemmNonCandidate}
    Let $n\in\mathbb{N}$ be $B$-critical with $b=\baseB n$, and
    let $\min B < m \leq b$. Let $\delta < \Om$ be such that for all $m' < m$ holds $o_B(m') < \vartheta(\alpha_B(n) + \de)$. If $m$ is not a candidate for $n_*$, then also $o_B(m) < \vartheta(\alpha_B(n) + \de)$.
\end{lemma}

\begin{lemma}\label{lemmO*}
\begin{enumerate}
    \item If $n_*$ does not exist and $\alpha_B(n) > 0$, then $\beta_B(n) = 0$.
    \item If $n_*$ does not exist and $\alpha_B(n) = 0$, then $\beta_B(n) = 2$.
    \item If $n_*$ exists and $\alpha_B(n_*) > \alpha_B(n)$, then $\beta_B(n) = o_B(n_*)$.
    \item If $n_*$ exists and $\alpha_B(n_*) = \alpha_B(n)$, then $\beta_B(n) = \zeta + 1$, where $o_B(n_*) = \vartheta(\alpha_B(n) + \zeta)$.
\end{enumerate}
\end{lemma}

\begin{proof}
    Assume first that $n_*$ does not exist. Then every $m \leq b$ is not an $n_*$-candidate. If $\alpha_B(n) > 0$, then $o_B(\min B) = \om < \vartheta(\alpha_B(n))$, and by using Lemma~\ref{lemmNonCandidate} inductively we get that $o_B(b) < \vartheta(\alpha_B(n))$. So $\beta_B(n) = 0$. If $\alpha_B(n) = 0$, then $o_B(\min B) < \vartheta(2)$, and by the same procedure $o_B(b) < \vartheta(2)$. Furthermore from monotonicity we get $\vartheta(1)\leq o_B(b) < \vartheta(2)$.

    Next suppose that $\alpha_B(n_*) > \alpha_B(n)$. Then observe that for all $m'\leq n_*$ we have $o_B(m') \leq o_B(n_*) < \vartheta(\alpha_B(n) + o_B(n_*))$. By using Lemma~\ref{lemmNonCandidate} inductively for $n_* <  m \leq b$, we get that $o_B(b) < \vartheta(\alpha_B(n) + o_B(n_*))$. In order to prove that $\beta_B(m) = o_B(n_*)$, it suffices to show that for all $\ga < o_B(n_*)$ we have that 
    \[\vartheta(\alpha_B(n) + \ga) < o_B(n_*) < o_B(b).\]
     This follows easily by using Proposition~\ref{propCompareTheta}.

     Finally, assume that $\alpha_B(n_*) = \alpha_B(n)$. We can again use Lemma~\ref{lemmNonCandidate} to get $o_B(b) < \vartheta(\alpha_B(n) + \zeta + 1)$. Considering $o_B(n_*)$, it is clear that $\zeta + 1$ is the minimal candidate for $\beta_B(n)$, so $\beta_B(n) = \zeta + 1$.
\end{proof}

\begin{proposition}\label{propPresO}
Let $C$ be a good successor of $B$, and write $\ug$ for $\ug_B^C$. 
Then for every $n\in\mathbb{N}$ with $b=\baseB n$ holds 
\begin{enumerate}
\item $o_B(n) = o_C(\ug n)$.\label{pres_o}
\item $O_B(n) = O_C^c(\Chgbases b c n)$ if $c\geq\ug b $.\label{pres_O^b}
\item $O_B(n) = O_C(\ug n)$. \label{pres_O}
\item $\alpha_B(n) = \alpha_C^c(\Chgbases bc n)$ if $c\geq \ug b$.\label{pres_alpha^b}
\item $\alpha_B(n) = \alpha_C(\ug n)$. \label{pres_alpha}
\end{enumerate}
\end{proposition}

\begin{proof} 
If $c$ is the witness for $\ug n$, then $c\geq \ug b$ and
$\basep C {\ug n } = c$ by Lemma~\ref{lemmalternativedefupgr} and Lemma~\ref{lemmUp(ba+r)}, respectively. So \ref{pres_O} and \ref{pres_alpha} follow from \ref{pres_O^b} and \ref{pres_alpha^b}, respectively. 

Moreover it is easy to see that \ref{pres_alpha^b} follows from \ref{pres_O^b}. Indeed, if $n=b^2u+bv+w$ then $\Chgbases bc n = c^2u' + c\cdot\ug v + \ug w$ and $\ug v,\ug w < \ug b\leq c$. Then $O_B(b^2u) = O_C^c(\Chgbases bc b^2u) = O_C^c(c^2u')$, which implies $\alpha_B(n) = \alpha_C^c(\Chgbases bc n)$.

We now prove \ref{pres_o} and \ref{pres_O^b} simultaneously by induction on $n$. The case $n<\min B$ is trivial, so assume $n\geq \min B$. It is easy to see that \ref{pres_O^b} follows from the induction hypothesis for \ref{pres_o}. Indeed, 
write $n=b^{e_1}\cdot a_1+\dots+b^{e_k}\cdot a_k$. Then 
\begin{align}\label{eqC-repr}
\Chgbases bc n = c^{\Chgbases bc e_1}\cdot\ug a_1+\dots + c^{\Chgbases bc e_k}\cdot \ug a_k. \tag{$\star$}
\end{align}
By applying the induction hypothesis of \ref{pres_o} to each $a_i$ and that of \ref{pres_O^b} to each $e_i$,
\begin{align*}
O_B^b(n) &= \Om^{O_B^b(e_1)} o_B(a_1) +\dots + \Om^{O_B^b(e_k)}o_B(a_k)\\
&= \Om^{O_C^c(\Chgbases bc e_1)}o_C(\ug a_1) + \dots + \Om^{O_C^c(\Chgbases bc e_k)}o_C(\ug a_k) = O_C^c(\Chgbases bc n).
\end{align*}
In the last inequality we use that (\ref{eqC-repr}) is written in $c$-normal form, which follows from $\ug b\leq c$.

We now prove that \ref{pres_o} holds for $n$. By the above we can assume (together with the induction hypothesis for all items) that \ref{pres_O^b} to \ref{pres_alpha} holds for $n$.

\begin{Cases}

\item ($n = b \in B$).
If $n=\min B$, then $\ug n=\min C$ so $o_B(n) = o_C(\ug n) = \om$.
Otherwise, let $d$ be the predecessor of $b$ in $B$.
By induction hypothesis, $o_B(b-d) = o_C(\ug(b-d))$.
Let $c = \basep C{\ug(b-1)}$.
By Lemma~\ref{lemmUp(ba+r)}\ref{lemmUp(ba+r)b-d},
\[o_C(\ug b) = o_C(\ug b -c) \cdot 2 =o_C(\ug (b-d)) \cdot 2= o_B (b-d) \cdot 2=o_B(b).\]

\item ($n$ is not $B$-critical).
Write $n=ba+r$ with $0<r<b$.
By Lemma~\ref{lemmUp(ba+r)}\ref{lemmUp(ba+r)ba+r}, $\ug n = \ug ba +\ug r $.
The induction hypothesis then yields
\[o_C(\ug n) = o_C(\ug ba) +o_C(\ug r) = o_B(  ba) +o_B( r) = o_B( n) .\]

\item ($n\notin B$ is $B$-critical).
Write $n=b^2u+bv$ and $\ug n = d^2 u' + d\cdot\ug v$. Observe that $\widetilde{\ug v} = \ug \tilde v$. 

By \ref{pres_alpha} we have that $\alpha_B(n) = \alpha_C(\ug n) \eqqcolon \al$. Assume that $\ug b = c\leq d$. It is clear that 
\[o_B(b) = o_C(c)\leq o_C(d) < \vartheta(\alpha + \beta_C(\ug n)),\]
so $\beta_B(n)\leq \beta_C(\ug n)$.

The rest of the proof is focused on showing that either $\beta_B(n) = \beta_C(\ug n)$, or that $\beta_B(n) + \ga = \beta_C(\ug n)$ for some ordinal $\ga$. In the latter case, there will be a non-zero term $\om^{o_B(\tilde v)}>\ga$, where $\tilde v$ is either $v-2$ or $v$, which will ensure that
\[\beta_B(n) + \om^{o_B(\tilde v)} = \beta_C(\ug n) + \om^{o_B(\tilde v)} = \beta_C(\ug n) + \om^{o_C(\ug\tilde v)}.\]
In the last equality we have used the induction hypothesis on $\tilde v$. It then follows that $\zeta_B(n) = \zeta_C(\ug n)$, so $o_B(n) = o_C(\ug n)$.

First we claim that for every $C$-critical $m$ strictly between $c$ and $\ug n$, we have that $O_C(m) < O_C(\ug n)$. For suppose not. Then, because $O_C$ is monotone between $d$ and $\ug n$, we must have $m<d$. Let $\basep C m = e$ for some $c\leq e < d$, then
\[O_C^e(m) = O_C(m) \geq O_C(\ug n) = O_C^e(\Chgbases be n).\]
By the monotonicity of $O_C^e$ (Lemma~\ref{lemmO_fMon}), we have that $\Chgbases be n \leq m < S_C(e)\leq d$. Since $\ug b\leq e$, Lemma~\ref{lemmalternativedefupgr} implies that the witness for $\ug n$ is less than or equal to $e$, which is less than $d$. A contradiction.

Next we will show by induction on $m$ that for $m \leq d$, we have $o_C(m) < \vartheta(\alpha + \beta_B(n) + \ga)$ for some ordinal $\ga < \om^{o_B(\tilde v)}$. Here $\om^{o_B(\tilde v)}$ is considered to be zero when $\tilde v = 0$.
Then, by setting $m=d$, we get thet $\beta_C(\ug n) \leq \beta_B(n)+\ga$ and the desired result follows.

If $m\leq c$, then this holds by the definition of $\beta_B(n)$ and because $o_B(b) = o_C(c)$. Furthermore, if $m$ is not an $(\ug n)_*$-candidate, then the inductive step follows from Lemma~\ref{lemmNonCandidate}. So assume that $m$ is a candidate for $(\ug n)_*$. Write $m=e^2u'' + ev''$. Because $O_C(m) < O_C(\ug n)$, we must have that $\alpha_C(m) = \alpha$ and $v'' < \ug v$.
We then have to show that
\[o_C(m) = \vartheta(\alpha + \beta_C(m) + \om^{o_C(\widetilde{v''})}) < \vartheta(\alpha + \beta_B(n) + \ga)\]
for some $\ga < \om^{o_B(\tilde v)}$. Since $\widetilde{v'' } < \widetilde{\ug v } = \ug \tilde v$, it suffices to prove that $\beta_C(m) < \beta_B(n) + \om^{o_B(\tilde v)}$.
By the induction hypothesis, there is some $\ga' < \om^{o_B(\tilde v)}$ such that
\[o_C(e) < \vartheta(\alpha + \beta_B(n) + \ga').\]
Then by the minimality of $\beta_C(m)$, we get that $\beta_C(m)\leq \beta_B(n) + \ga'$, and we are done.

\end{Cases}
\end{proof}

\begin{theorem}\label{theoTerm}
If $\mathcal B$ is any dynamical hierarchy and $m\in\mathbb N$, then there is $i\in \mathbb N$ such that $\G^{\mathcal B}_i(m) = 0$.
\end{theorem}

\begin{proof}
Let $o_i$ denote $o_{\mathcal B_i}$, $\upgrade i$ denote $\upgrade{\mathcal B_i}^{\mathcal B_{i+1}}$.
Let $I$ be the length of the Goodstein sequence starting on $m$.
Then, if $ \G^{\mathcal B}_i(m)>0$ we have that
\[ o_{i+1}(\G^{\mathcal B}_{i+1}(m)) = o_{i+1}(\upgrade i \G^{\mathcal B}_i(m)-1) < o_{i+1}(\upgrade i \G^{\mathcal B}_i(m) ) = o_{i }( \G^{\mathcal B}_i(m) ) , \]
where the first inequality is by Proposition~\ref{propMonO} and the second by Proposition~\ref{propPresO}.
It follows that $(o_i(\G^{\mathcal B}_i(m)))_{i<I}$ is a decreasing sequence of ordinals, hence it is finite and we must have $o_{I-1}(\G^{\mathcal B}_{I-1}(m)) = 0$.
\end{proof}
 
\section{The canonical dynamical hierarchy}

Our goal is to define a dynamical hierarchy which will yield a Goodstein principle independent of $\sf KP$. In \cite{fernandez2025fractal}, such a base hierarchy was obtained by iteratively taking so-called `greedy successors'. We define another type of successor, the ouroboros successor, which will be more suitable for our purposes.

\begin{definition}\label{defCanon}
Given a base hierarchy $B$ and $i\in\mathbb N\setminus\{0\}$, we define a new base hierarchy $B_{+i} = \bigcup_{n <\infty} B_{+i}^n  $, where $B_{+i}^n$ is defined recursively as follows.
Assume inductively that $B_{+i}^m $ has been defined and is finite for $m<n$.

First, set $B_{+i}^n   = \{ \min B  + 1\}$ if $n\leq \min B$.

For $n>\min B$, let $b = \baseB n$ and $c=\max B_{+i}^{n-1}$.
\begin{enumerate}

\item If $n$ is not $B$-critical, then $B_{+i}^{n } =   B_{+i}^{n-1} $.

\item If $n \in B$, then let $b = \baseB{n-1}$ and set
\[B_{+i}^{n } = B_{+i}^{n-1} \cup \{ ca  \}     ,\]
where $a$ is minimal such that $\Chgbases bc(n-1) < ca$.

\item If $n \notin B$ is $B$-critical, let $b=\baseB{n}$. Then,
\[B_{+i}^{n} = B_{+i}^{n-1} \cup \{d_j\}_{j = 1}^i   ,\]
where we define $d_j =  d_j(n,i)$ recursively by $d_0 =  c $ and
\[ d_{j +1 }   = \Chgbases {b}{d_j} n  .\]

\end{enumerate}
\end{definition}
We define the {\em canonical dynamical hierarchy} $\mathcal C=(C_i)_{i\in\mathbb N}$ to be the unique dynamical hierarchy with $C_0=\{3\}$ and $C_{i+1} = (C_i)_{+(i+2)}$. We will show that $\mathcal C$ is in fact a dynamical hierarchy, namely that it consists of good successors, and that the corresponding Goodstein process is independent of $\sf KP$.
\begin{lemma}\label{lemmB_(+i)found}
    Let $B$ be a base hierarchy and $n\in\N$. Let $C_n = B_{+i}^n$ and $c_n = \max C_n$. Write $\ug_n$ for $\ug_B^{C_n}$.
        Then $\ug_n n<\infty$, and $c_n$ is the witness.
\end{lemma}

\begin{proof}
    We proceed by induction on $n$. The case $n\leq \min B$ is trivial, so assume otherwise.
    We first prove the following claim.

    \begin{statement}
    Whenever a base is added to $C_n$ that was not yet in $C_{n-1}$, this base is bigger than $c_{n-1}$ and therefore must be $c_n$. Moreover, $c_{n-1}\mid c_n$ and $c_{n-1}\leq \ug_{n-1} (n-1) < c_n$.
    \end{statement}

    In the case where $n\in B$ we have by Lemma~\ref{lemmMonUg}\ref{lemmMonUgLowerBnd} that $c_{n-1}\leq \Chgbases b{c_{n-1}}(n-1)  < c_{n-1}a = c_n$, where $b=\baseB {n-1}$. It is clear that $c_{n-1}\mid c_n$, and by the induction hypothesis we have that $\ug_{n-1}(n-1) = \Chgbases b{c_{n-1}}(n-1)$.
    
    Otherwise $n\notin B$ is $B$-critical. Then $b=\baseB{n-1}=\baseB n$, and by the induction hypothesis together with Lemma~\ref{alternativedefwitness} we have $c_{n-1}\geq \ug_{n-1} b$. Then $d_1 = \Chgbases b{c_{n-1}}n > \Chgbases b{c_{n-1}} b = c_{n-1}$, and $c_{n-1}\mid d_1$ by Lemma~\ref{lemmStructureChgbases}\ref{lemmStructureChgbasesDiv}. Iterating, we see that $d_{j+1}$ is a proper divisor of $d_j$ for each $j < i$. Moreover in this case $\ug_{n-1}(n-1) = \Chgbases b{c_{n-1}}(n-1) < d_1$. This establishes the claim.

    By our claim we know that $C_{n-1}$ is the restriction of $C_n$ to the interval $[0, \ug_{n-1}(n-1)]$. Then by Lemma~\ref{lemmrestrictbasehier}, $\ug_n x = \ug_{n-1}x < \infty$ for all $x < n$, and the witnesses for both upgrades are equal. Let $b=\baseB n$.
    \begin{Cases}
    \item ($n$ is not $B$-critical). Then $c_{n} = c_{n-1}$, and
    \[\ug_n (n-1) = \Chgbases {b}{c_n}(n-1) < \Chgbases {b}{c_n}n < \infty = S_{C_n}(c_n).\]
    By Lemma~\ref{lemmMonUg}\ref{lemmMonUgWitnComp}, $c_n$ is the witness for $\ug_n n$.

    \item ($n$ is $B$-critical).
    Then by the claim and Lemma~\ref{lemmMonUg}\ref{lemmMonUgLowerBnd},
    \[\ug_n(n-1) < c_n \leq \Chgbases b{c_n} n < \infty = S_{C_n}(c_n).\]
    By the definition of the upgrade, we have to prove that $c_n$ is the minimal base satisfying this. By Lemma~\ref{lemmMonUg}\ref{lemmMonUgWitnComp} it suffices to rule out $c_{n-1}$ and potentially the bases between $c_{n-1}$ and $c_n$.
    \begin{Cases}
        \item ($n = b\in B$).
        Then there are no other bases between $c_{n-1}$ and $c_n$, and $\Chgbases b{c_{n-1}}n = c_{n-1}\leq \ug_n(n-1)$.
        \item ($n\notin B$).
        Then $c_{n-1}$ and the bases between $c_{n-1}$ and $c_n$ are of the form $d_j$ for $j<i$, and $\Chgbases b {d_j}n = d_{j+1} = S_C(d_j)$.
    \end{Cases}
    \end{Cases}
\end{proof}

\begin{lemma}\label{lemmGoodSuccB(+i)}
Let $B$ be a base hierarchy and $C=B_{+i}$ for some $i$.
Then $C$ is a good successor of $B$.
\end{lemma}

\begin{proof} 
The fact that $C$ is a base hierarchy follows from the claim in the previous lemma. It is also clear that $\min C = \min B + 1$.

For $n\in\N$, let $C_n = B_{+i}^n$. Write $\ug$ for $\ug_B^C$ and $\ug_n$ for $\ug_B^{C_n}$.
By Lemma~\ref{lemmrestrictbasehier}, we have that $\ug n = \ug_n n < \infty$.

To verify the final condition, suppose that $\min B <n\in B$. Then the witness for $\ug(n-1)$ is $c_{n-1}$, and we have to show that there are no multiples of $c_{n-1}$ in the interval $[\ug(n-1), c_n)$. But this holds because $c_n$ is chosen as the minimal multiple of $c_{n-1}$ above $\Chgbases b{c_{n-1}} (n-1) = \ug_{n-1}(n-1) = \ug(n-1)$, where $b=\baseB{n-1}$.
\end{proof}

Given a base hierarchy $B$, consider $C =B_{+i}$ for some $i$. Then $C$ is the union of $C_m \coloneqq B_{+i}^m$. If $n\notin B$ is $B$-critical, then by Definition~\ref{defCanon}, there are corresponding bases $d_j =d_j(n,i)$ in $C$, for $0\leq j \leq i$. Definition~\ref{defCanon} says that $d_{j+1} = \Chgbases {b}{d_j} n$ holds in some $C_m$, with $b=\baseB n$. By Lemma~\ref{lemmrestrictbasehier}, this remains true in $C$. 
We record the following properties for later use in the proof of independence in Section $7$.

\begin{lemma}\label{lemmCanonProp}
Let $n\notin B$ be $B$-critical with $b=\baseB n$. Write $\ug$ for $\ug_B^C$ and $d_j$ for $d_j(n,i)$.
    \begin{enumerate}
        \item $d_i$ is the witness for $\ug n$ and $d_0$ is the witness for $\ug (n-b)$.\label{lemmCanonPropWitness}
        \item If $n+b\in B$, then there are no bases of $C$ lying strictly between $d_i$ and $\ug(n+b)$.\label{lemmCanonPropBases}
    \end{enumerate}
\end{lemma}
\begin{proof}
    Let $c_m=\max C_m$ for each $m\in\N$. 
    From the proof of Lemma~\ref{lemmB_(+i)found}, we see that $c_n$ is one of the elements $d_j$, and for each $j$ we have that $d_j$ is a proper divisor of $d_{j+1}$. Therefore $c_n = d_i$. By Definition~\ref{defCanon}, $d_0=c_{n-1} = c_{n-2}=\dots = c_{n-b}$. Now \ref{lemmCanonPropWitness} follows from Lemma~\ref{lemmB_(+i)found} together with Lemma~\ref{lemmrestrictbasehier}. Suppose that $n+b\in B$. Then $\ug(n+b) = c_{n+b}$ by Lemma~\ref{lemmB_(+i)found} and Lemma~\ref{lemmrestrictbasehier}. Furthermore $d_i=c_n = c_{n+b-1}$, and by the claim in Lemma~\ref{lemmB_(+i)found} together with Definition~\ref{defCanon}, $c_{n+b}$ is the successor of $c_{n+b-1}$ in $C$. This establishes \ref{lemmCanonPropBases}.
\end{proof}

\section{Fundamental Sequences}

 We recall some notions from \cite{BuchholzOrd} and \cite{FWTheta}.

\begin{definition}\label{defTau}
Let $\xi<\ve_{\Om+1}$ be in $\Om$-normal form.
The {\em cofinality} of $\xi$, denoted $\tau(\xi)$, is given recursively by
\begin{enumerate}

\item $\tau(0) = 0$ and $\tau(\zeta+1) = 1$,

\item $\tau (\Om^\al \be+\ga) = \tau  (\ga)$ if $\ga>0$,

\item $\tau (\Om^\al \be ) = \be$ if $\be$ is a limit,

\item $\tau (\Om^\al (\be+1) ) = \tau(\al)$ if $\al$ is a limit, and

\item $\tau(\Om^ {\al+1}(\be+1)) = \Om$.

\end{enumerate}
\end{definition}

\begin{definition}\label{defFS}
    For $\xi < \ve_{\Om+1}$ in $\Om$-normal form and $\theta < \Om$, we define
    \begin{enumerate}
        \item $\fs 0\theta = \fs 1\theta = 0$,
        \item $\fs{(\Om^\al\be + \ga)}{\theta} = \Om^\al\be + \fs\ga\theta$ if $0 <\ga < \Om^\al$,
        \item $\fs{\Om^\al\be}{\theta} = \Om^\al\theta$ if $\be\in\li$,
        \item $\fs{(\Om^\al(\beta + 1))}{\theta} = \Om^\al\be + \fs{\Om^\al}\theta$ if $\be > 0$,
        \item $\fs{\Om^\al}{\theta} = \Om^{\fs \al\theta}$ if $\al\in\li$, and
        \item $\fs {\Om^{\al + 1}}{\theta} = \Om^\al\theta$.
    \end{enumerate}
\end{definition}

\begin{proposition} \label{propCompareFS}
Let $\zeta\in \ve_{\Om+1}$ and $\theta,\eta < \Om$.
    \begin{enumerate}
        \item $\mco{\fs\zeta\theta}\leq \max(\mco\zeta, \theta)$.\label{propCompareFSmco}
        \item If $\eta < \theta$ and $\zeta\in\li$, then $\fs\zeta\eta < \fs\zeta\theta$. \label{propCompareFSmon}
    \end{enumerate}
\end{proposition}

\begin{proof}
    Induction on $\zeta$. These properties are also stated in \cite{BuchholzOrd}.
\end{proof}

\begin{definition}\label{defFixJump}
    We introduce the following notation.\vspace{2pt}
    \begin{enumerate}[itemsep=4pt]

        \item $\FIX = \{\xi < \ve_{\Om+1} \mid \mco{\fs \xi 1} < \mco\xi = \tau(\xi) = \B(\ga) \textit{ for some } \ga >\xi\}$

        \item $\JUMP = \FIX\cup \su\cup\{0\}$

        \item $\vartheta^* (\xi) = \begin{cases}
\B(\zeta)&\text{if $\xi =\zeta+1$,}\\ 
\tau(\xi)&\text{if $\xi\in \FIX$,}\\
0 &\text{otherwise.}
 \end{cases}$

 \item For $\xi=\Om\al+\be$ with $\be<\Om$, set \vspace{5pt}\newline
$\check \xi=\begin{cases}
\Om\al&\text{if $\vartheta^*(\zeta)>0$,}\\ 
\xi &\text{otherwise.}
 \end{cases}$

    \end{enumerate}
\end{definition}

We recall Buchholz's fundamental sequences~\cite{BuchholzOrd}. It will be convenient to define them for all elements of $\vartheta[\ve_{\Om+1}]$. Note that $\check\xi = 0$ is equivalent to $\xi\in\Om\cap\JUMP$.

\begin{definition}\label{defCFS}
Let $\xi\in \vartheta[\ve_{\Om+1}]$ and $\iota <\om$.
Define $\fsc \xi\iota \in \vartheta[\ve_{\Om+1}]$ by
\begin{enumerate}

\item 
$\fsc 0\iota = 0$.

\item
$\fsc{(\om^\al+\be)}\iota = \om^\al+\fsc\be\iota$ if $0<\be<\om^{\al+1}$.

\item If $\xi=\vartheta(\zeta)$, then let $\tau = \tau(\check\zeta)$, and:
\begin{enumerate}

\item If $\zeta\in\Om\cap \JUMP$ then $\fsc{\B( \zeta)}\iota  = \vartheta^*(\zeta)\cdot\iota$.

\item If $0<\tau<\Om$, then
$\fsc{\B( \zeta)}\iota = \B \big ( \fs{\check \zeta} {\fsc \tau\iota} +\vartheta^*(\zeta) \big )$.

\item If $\tau=\Om$, then
$
\fsc{\B(\zeta)}0=\vartheta^*(\zeta)
$ and
$\fsc{\B(\zeta)}{\iota+1}=\vartheta( \fs{\check\zeta} {\fsc{\xi}\iota} )$.

\end{enumerate}
\end{enumerate}
\end{definition}

It is shown in \cite{BuchholzOrd} that the usual property of fundamental sequences hold, namely if $\xi\in\vartheta[\ve_{\Om+1}]\cap\li$, then $\fsc{\xi}n$ is strictly increasing in $n$ and converges to $\xi$. We extend this to hold for $\xi=\vartheta[\ve_{\Om+1}]$ by defining $\fsc{\vartheta[\ve_{\Om+1}]}n = \vartheta(\Om_n)$.

\begin{lemma}\label{lemmZetaIota}
Let $\zeta < \ve_{\Om+1}$, $0<\theta<\Om$, $\xi\in \vartheta[\ve_{\Om+1}]$ and $0<\iota<\om$. 

\begin{enumerate}
\item 
If $\zeta > 1$, then $\fs\zeta\theta > 0$.

\item\label{itZetaIotaInf}
If $\zeta$ is infinite then either $\fs\zeta\theta=\theta$, or else $\fs\zeta\theta$ is infinite.

\item\label{itZetaIotaUnc} If $\zeta$ is uncountable then either $ \zeta[\theta]=\theta$, or else $\zeta[\theta]$ is uncountable.

\item
If $\xi>1$, then $\fsc\xi\iota > 0$.

\item\label{itXiIotaInf}
If $\xi$ is infinite then either $\fsc\xi\iota = \iota$, or else $\fsc\xi\iota$ is infinite.

\end{enumerate}
\end{lemma}

\begin{proof}
We prove the first three claims by simultaneous induction on $\zeta>0$. Write $\zeta=\Om^\eta\al+\ga$, with $\al$ additively indecomposable.

If $\ga>0$ then $\zeta[\theta] = \Om^\eta\alpha + \fs\ga\theta$ is uncountable if $\zeta$ is uncountable, infinite if $\zeta$ is infinite, and greater than zero. Now assume that $\ga = 0$.

If $\al$ is a limit, then $\zeta[\theta]=\Om^\eta\theta$, which is uncountable if $\zeta$ is, else it is equal to $\theta>0$.

Otherwise, $\al=1$.
If $\eta$ is a limit, then $\zeta[\theta] =\Om^{\eta[\theta]}$ is uncountable since $\eta[\theta] > 0$.
Otherwise, $\eta =\delta+1$ is a successor and $\zeta[\theta]$ is uncountable unless $\delta=0$.
But then, $\zeta[\theta]=\theta>0$.

Now we prove the last two claims by simultaneous induction on $\xi$. If $\xi = \om^\alpha + \beta$ with $0<\beta<\om^{\alpha+1}$, then $\fsc\xi\iota = \om^\alpha + \fsc\beta\iota$ is infinite if $\xi$ is, and greater than zero if $\xi > 1$. So assume further that $\xi = \vartheta(\zeta)$ for some $\zeta$. Then $\xi > 1$ only happens when $\zeta > 0$ and $\xi$ is infinite, so that we only have to consider the last claim. 

Note that $\vartheta^*(\zeta)$ is either zero or of the form $\vartheta(\zeta')$, in which case it is either $1$ or infinite.
In the case where $\fsc\xi\iota$ is of the form $\vartheta^*(\zeta)\cdot\iota$, we have that $\vartheta^*(\zeta) > 0$, and the assertion follows. 

Now consider the case where $\fsc\xi\iota = \vartheta(\fs{\check\zeta}{\fsc\tau\iota} + \vartheta^*(\zeta))$. 
It suffices to check that 
$\fs{\check\zeta}{\fsc\tau\iota} + \vartheta^*(\zeta)>0$.
For this not to be the case we would need that $\vartheta^*(\zeta) = 0$, so that $\check\zeta = \zeta$.
Since $\zeta\notin\Om\cap\JUMP$ we have that $\check\zeta$ is infinite. Then by the first item, $\fs{\check\zeta}{\fsc\tau\iota} = 0$ can only occur if $\fsc\tau\iota = \fsc{\tau(\zeta)}\iota =0$. Since $\tau(\zeta)\leq \mco\zeta < \xi$, we can apply the induction hypothesis to see that in this case $\tau \leq 1$. Since $\xi > 1$, it must be that $\tau = 1$ and that $\zeta$ is a successor. But this contradicts the fact that $\vartheta^*(\zeta) = 0$.

Now suppose that $\tau(\check\zeta) = \Om$.
Then $\fsc\xi 1 = \vartheta(\fs{\check\zeta}{\vartheta^*(\zeta)})$ is either one or infinite. In any case, by the first claim $\fs{\check\zeta}{\fsc\xi 1} > 0$, so $\fsc\xi 2$ is infinite. Repeating this argument shows that $\fsc\xi \iota$ is infinite for every $\iota \geq 2$.
\end{proof}

\begin{lemma}\label{lemmCFSsuccli}
    If $\alpha<\vartheta[\ve_{\Om+1}]$, then
    \begin{enumerate}
        \item If $\alpha<\om^\al$, then $\al\notin\FIX$.\label{lemmCFSfix}
        \item If $\alpha\in\li$, then $\fsc{\om^\al}\iota \leq \om^{\fsc\al\iota+1}$.\label{lemmCFSli}
        \item If $\al\in\su$, then $\fsc{\om^{\alpha}}\iota = \om^{\alpha-1}\cdot\iota$.\label{lemmCFSsucc}
    \end{enumerate}
\end{lemma}
\begin{proof}
    For \ref{lemmCFSfix}, assume $\al$ is in the image of $\vartheta$, so that $\al=\om^\be$. By assumption $\be < \om^\be$, hence by Lemma~\ref{lemmVarthetaLessOm}, $\al = \om^\be$ is equal to either $\vartheta(\be)$ or $\vartheta(\be - 1)$. By the injectivity of $\vartheta$, we cannot have $\al = \vartheta(\ga)$ with $\ga > \al\geq \beta$.

    If $\al=\om^\al$, then \ref{lemmCFSli} holds trivially. So assume $\al\in\li$ and $\al < \om^\al$. Then $\vartheta(\al) = \om^\al$. One checks that 
    \[\fsc{\vartheta(\al)}\iota = \vartheta(\fs{\check\al}{\fsc\al\iota} + \vartheta^*(\al)) = \vartheta(\fsc\al\iota) \leq\om^{\fsc\al\iota + 1}.\]
    For \ref{lemmCFSsucc}, if $\al+1$ is not of the form $\de+n$ where $\om^\de=\de$ and $n<\om$, then
    \[\fsc{\om^{\al+1}}{\iota}=\fsc{\vartheta(\al+1)}{\iota} = \vartheta^*(\al+1)\cdot \iota = \vartheta(\al)\cdot\iota = \om^\al\cdot\iota.\]
    Otherwise suppose that $\al+1$ is of the form $\de+n$. If $n>1$, we can make the same calculation as above. Suppose further that $\al=\de$. By the surjectivity of $\vartheta$, we have $\de=\vartheta(\ga)$, and by Lemma~\ref{lemmVarthetaLessOm}, $\ga \geq\Om$, which implies $\de\in\FIX$. Then
    \[\fsc{\om^{\al+1}}\iota = \fsc{\vartheta(\de)}\iota = \vartheta^*(\de)\cdot\iota = \de\cdot \iota =\om^\al\cdot\iota.\]
\end{proof}

\begin{definition}\label{defFFun}
For $i<\om$ and $\al<\vartheta[\ve_{\Om+1}]$, define $\fsi\al i$ recursively by $\fsi\al 0 = \al$ and $\fsi\al {i+1} = \fsc {\fsi \al i}{i+1} $.
Define $F_\al(n)$ to be the least $\ell$ such that
$\fsi{\fsc{\al} n} \ell = 0$.
\end{definition}

For each $\alpha$, the function $F_\alpha$ is total since $ \fsi{\fsc\alpha n} {i+1} < \fsi{\fsc\al n} i$ whenever the right-hand side is not zero. However, the proof-theoretic strength required to establish totality grows with $\alpha$. In particular for $\alpha=\vartheta[\ve_{\Om+1}]$, the totality of $F_\al$ is not provable in $\sf KP$.
In fact, a more general claim holds: it is a special case of a general principle of Cichon et al.~\cite{BCW} adapted to Buchholz's system of fundamental sequences~\cite{BuchholzOrd,FWTheta}.

\begin{theorem}\label{theoKPInc}
Let $f$ be a computable function, and let $T$ denote one of the systems $\mathsf{RCA}_0 + (\Sigma^{0}_{n}) - \mathrm{IND}$, $\mathsf{ACA_0}$, $\mathsf{ATR_0}$, $\mathsf{KP}^{-}\omega + (\Pi_{n})-\mathsf{IND}$, or $\mathsf{KP}$. If $T$ proves the totality of $f$, i.e.\
$T\vdash \forall x\exists y (y=f(x))$
\footnote{The expression $y=f(x)$ should be understood as $\varphi_f(x,y)$, where $\varphi_f(x,y)$ is a $\Sigma^0_1$ definition of the graph of $f$.}, then $\exists m \ \forall n>m \ \big ( f (n) < F_\al(n) \big)$, where $\al$ is the proof-theoretic ordinal of $T$:
\begin{enumerate}
\item
If $T\equiv \mathsf{RCA}_0 + (\Sigma^{0}_{n}) - \mathrm{IND}$, then $\al = \om_{n+1}$.
\item 
If $T \equiv \sf ACA_0$, then $\alpha = \vartheta(\Om) = \ve_0$.
\item 
If $T \equiv \sf ATR_0$, then $\al = \vartheta(\Om^2) = \Gamma_0$.
\item
If $T \equiv \mathsf{KP}^{-}\omega + (\Pi_{n})-\mathsf{IND}$, then $\al = \vartheta((\Om_n)^\om)$.
\item 
If $T \equiv \sf KP$, then $\al = \vartheta[\ve_{\Om+1}]$.
\end{enumerate}
\end{theorem}

The following proposition is the fundamental tool for majorizing the function $F$.

\begin{proposition}[\cite{FSGoodstein,FWTheta}]\label{propFunMajor}
Suppose that $(\xi_i) _{i \leq I} < \vartheta[\ve_{\Om+1}]$ is such that for all $i<I $,
\[\fsc{\xi_{i}}{i+1}\leq \xi_{i+1} \leq \xi_i.\]
Then, for all $i\leq I$, $\xi_i\geq \fsi{\xi_0}i$.
\end{proposition}

\section{Independence}



The following lemma roughly says that, given $n>0$, applying the square fundamental sequences to $O_B^b(n)$ (where we possibly ignore a small enough power of $\Om$) is the same as replacing $n$ by a smaller number. 
This result will be used to bound occurences of $\alpha_B$ in the independence proof.

\begin{lemma}\label{lemmUCFS}
Let $B$ be a base hierarchy.
Let $n>0$, $b\in B$ and $O_B^b(n) = \Om^d\xi$ with $d<\min B$. Suppose that $0<\iota = o_B(c)$ for some $c<b$, where either $\iota <\min B - d$ or $\om\leq \iota< \tau(\xi)$.

Then, there is $n' <n$ such that
\[ O_{B}^b( n' ) =  \Om^d(\fs{\xi} \iota) .\]
Moreover, 
\begin{enumerate}
    \item If $b^e\mid n$ with $e>0$, then $n'\leq n-b^{e-1}(b-c)$.
    \item If $n>b$ and $\xi\geq \Om$, then $n'\geq b$.
\end{enumerate}
\end{lemma}

\begin{proof}
Write $n=_b b^ea+r$ and $\xi=\Om^\eta\al+\ga$. We construct $n'$ by induction on $n$. One checks that all of the items at the bottom of the statement are satisfied in each case.

\begin{Cases}
\item ($r>0$).
Then $O_B^b(r) = \Om^d\ga > 0$.
We apply the induction hypothesis to find suitable $r' < r$ and set  $n' = {b}^{e} a +r'$.

\item ($r = 0$).
Then, $\ga=0$ and $o_B(a)=\al$.
\begin{Cases}
\item ($\min B\nmid a>1$). 
Then by Lemma~\ref{standardPropOo}\ref{standardPropOo1}, $\alpha=\delta+1$ where $\de = o_B(a-1)$. By the induction hypothesis applied to $\Om^d\Om^{\eta} = O_B^b(b^e)$, we find suitable $r' < b^e$ such that $\Om^d\cdot\Om^{\eta}[\iota] = O_B^b(r')$. We then set $n' = b^e(a-1) + r'$.

\item ($\min B\mid a>1$).
By Lemma~\ref{standardPropOo}\ref{standardPropOo2}, $\al\in\li$. Then $\tau(\xi) = o_B(a)$, so by the assumption and the monotonicity of $o_B$, we have that $c < a$. We may then set $n' =  b^{e}c$.

\item ($a = 1$). We split into subcases for $e$.
\begin{Cases}
\item ($\min B \nmid e >0$).
In this case by Lemma~\ref{standardPropOo}\ref{standardPropOo5}, $d+\eta = \delta + 1$ is a successor, where $\de = O_B^b(e-1)$. If $\eta = 0$, then $\fs\xi\iota = \fs{\Om^0}\iota = 0$, and we can take $n'=0$. Otherwise $\Om^d\cdot\xi[\iota] = \Omega^\de\iota$ and we take $n'= b^{e-1}c$.

\item ($\min B\mid e > 0$).
Then $\eta\in\li$ so that $\fs\xi\iota = \Om^{\fs\eta\iota}$. If $\fs\eta\iota$ is infinite, then $\Om^d(\xi[\iota]) =\Om^{\fs\eta\iota}$.
Since $\tau(\xi) = \tau(\eta)$, we can use the induction hypothesis to find suitable $e'<e$ such that $O_B^b(e') = \fs\eta\iota$. We then set $n' = b^{e'}$.

Otherwise if $\fs\eta\iota$ is finite, then by Lemma~\ref{lemmZetaIota}\ref{itZetaIotaInf}, $\eta[\iota] =\iota$, so $d+\eta[\iota] <\min B$ by our assumptions. Thus we may set $n'=b^{d+\eta[\iota]} $.
In this case the first item follows from $d+\eta[\iota]<\min B$ while $e\geq \min B$.

\item ($e=0$).
Then $n=1$, and we can take $n'=0$.
\end{Cases}
\end{Cases}
\end{Cases}
\end{proof}

\begin{corollary}\label{corAlphaFS}
    Let $B$ be a base hierarchy. Let $n = b^2 u$ with $b\in B$ and $u>0$. Suppose that $0<\iota =o_B(c)$ for some $c<b$, where either $\iota<\min B - 1$ or $\om\leq \iota < \tau(\alpha_B^b(n))$.

    Then $\fs{\alpha_B^b(n)}\iota \leq \alpha_B^b(n - b^2) + \iota$.
\end{corollary}

\begin{proof}
    By applying the previous lemma with $d=1$, we find $n'$ such that
    \begin{enumerate}
    \item 
    $O_B^b(n') = \Om\cdot \fs{\alpha_B^b(n)}\iota$
    \item
    $b\leq n' \leq n - b(b-c) = b^2(u-1) + bc$.
    \end{enumerate}
    It follows that
    \[\Omega\cdot \fs{\alpha_B^b(n)}\iota = O_B^b(n')\leq O_B^b(b^2(u-1) + bc) = \Om\cdot(\alpha_B^b(n-b^2) + \iota)\]
\end{proof}

\begin{lemma}\label{lemmThetaStar}
Let $B$ be a base hierarchy and let $n\in\mathbb N$ be $B$-critical.
\begin{enumerate}
    \item If $o_B(n)$ contains a non-zero term $\om^{o_B(\tilde v)}$, then $\vartheta^*(\zeta_B(n)) = 0$. Otherwise,
    \item If $n_*$ does not exist and $\alpha_B(n) = 0$, then $\vartheta^*(\zeta_B(n)) = \om$.
    \item If $n_*$ does not exist and $\alpha_B(n) > 0$, then $\vartheta^*(\zeta_B(n)) = 0$.
    \item If $n_*$ exists and $\alpha_B(n_*) = \alpha_B(n)$, then $\zeta_B(n)\in \su$ and $\vartheta^*(\zeta_B(n)) = o_B(n_*)$.
    \item If $n_*$ exists and $\alpha_B(n_*) > \alpha_B(n)$, then $\zeta_B(n)\in \FIX$ and $\vartheta^*(\zeta_B(n)) = o_B(n_*)$.
\end{enumerate}
\end{lemma}

\begin{proof}
Assume first that
\[o_B(n) = \vartheta(\alpha_B(n) + \beta_B(n) + \om^{o_B(\tilde v)}),\]
where $\om^{o_B(\tilde v)} \neq 0$. Then it is clear that $\zeta_B(n)$ is not a successor, so it suffices to show that $\zeta_B(n)\not\in \FIX$. Suppose for a contradiction that $\zeta_B(n)\in \FIX$.

Since $\Om\mid \alpha_B(n)$, we have that $\tau \coloneqq\tau(\zeta_B(n)) = \beta_B(n) + \om^{o_B(\tilde v)}$. If $\beta_B(n) \geq \om^{o_B(\tilde v)}$, then $\tau$ is not additively indecomposable, which implies $\tau\neq \vartheta(\ga)$ for any $\ga$ and $\zeta_B(n)\notin\FIX$. 
So $\beta_B(n) < \om^{o_B(\tilde v)}$ and $\tau = \om^{o_B(\tilde v)}$. If $\om^{o_B(\tilde v)} > o_B(\tilde v)$, then by Lemma~\ref{propThetaMonEps}\ref{propThetaEps} and $\vartheta(\ga) = \tau =\om^{o_B(\tilde v)}$, we have that $\ga < \Om$. By Lemma~\ref{propThetaMonEps}\ref{propThetaMco}, $\ga < \vartheta(\ga) = \tau$, which again contradicts $\zeta_B(n)\in\FIX$.
So suppose further that $\omega^{o_B(\tilde v)} = o_B(\tilde v)$, in particular $\tilde v = v$ and $\tau = o_B(v)$. Then it follows from Lemma~\ref{standardPropOo}\ref{standardPropOo3} and $\zeta_B(n)\in\FIX$ that $v\notin B$ is $B$-critical, so that $\tau = \vartheta(\zeta_B(v))$. Furthermore, $\alpha_B(v) > \alpha_B(n)$ since otherwise $\zeta_B(v) < \zeta_B(n)$, contradicting $\zeta_B(n)\in\FIX$. 

    We now show that $\beta_B(n) \geq o_B(v) = \om^{o_B(v)}$, yielding a contradiction by our assumption on $\beta_B(n)$ made earlier. Let $b=\baseB n$.
    It is enough to show that for any $\ga < o_B(v)$ we have
    \[\vartheta(\alpha_B(n) + \ga) < \vartheta(\zeta_B(v)) = o_B(v) < o_B(b).\]
    For this we use Proposition~\ref{propCompareTheta}.
    Since $\alpha_B(v) > \alpha_B(n)$, we have $\alpha_B(n) + \ga < \zeta_B(v)$. Furthermore $\ga < o_B(v)$, and it follows from $\zeta_B(n)\in \FIX$ that $\mco{\alpha_B(n)} = \mco{\zeta_B(n)[1]} < \mco{\zeta_B(n)} = \tau = o_B(v)$. This establishes the first item.

    Next, assume that
    \[o_B(n) = \vartheta(\alpha_B(n) + \beta_B(n)).\]
    If $n_*$ does not exist and $\alpha_B(n) = 0$, then $\beta_B(n) = 2$ by Lemma~\ref{lemmO*} and it is clear that $\vartheta^*(\zeta_B(n)) = \om$.
    If $n_*$ does not exist and $\alpha_B(n) > 0$, then by Lemma~\ref{lemmO*}, $\zeta_B(n) = \alpha_B(n)$. Obviously $\alpha_B(n)\notin\su$, so we have to prove that $\alpha_B(n)\notin\FIX$. Suppose again for a contradiction that $\alpha_B(n)\in\FIX$.
    
    Then $\tau(\alpha_B(n)) = \alpha_B(n)^* = o_B(v) = \vartheta(\zeta_B(v))$ for some $B$-critical $v<b$ which is not a base. Again we have that $\alpha_B(v)\geq \alpha_B(n)$ since otherwise $\zeta_B(v) < \alpha_B(n)$. Now observe that since $v$ is $B$-critical, it has the same base as $v+1$. This implies that $\alpha_B(n)\leq \alpha_B(v) \leq \alpha_B(v+1)$ and $\mco{\alpha_B(n)} = o_B(v)<o_B(v+1)$. We see that $v+1$ is a candidate for $n_*$. Since $n_*$ does not exist, we get a contradiction.

Next assume that $n_*$ exists and $\alpha_B({n_*}) > \alpha_B(n)$.
Then $\beta_B(n) = o_B(n_*)$ by Lemma~\ref{lemmO*} and it is easy to see that $\zeta_B(n)\in\FIX$ with $\tau(\zeta_B(n)) = o_B(n_*)$, so that $\vartheta^*(\zeta_B(n)) = o_B(n_*)$.

Finally, if $\alpha_B({n_*}) = \alpha_B(n)$, then by Lemma~\ref{lemmO*} we get $\zeta_B(n) = \zeta_B(n_*) + 1$ so that $\vartheta^*(\zeta_B(n)) = \vartheta(\zeta_B(n_*)) = o_B(n_*)$.
\end{proof}

\begin{proposition}\label{propMajorizeFS}
Let $B$ be a dynamical hierarchy and $C=B_{+(i+1)}$ with $0<i<\min B - 1$.
Then for $n>0$,
\[o_{C}(\ug_B^C n-1) \geq \fsc{o_B(n)}i .\]
\end{proposition}

\begin{proof}
Write $\ug$ for $\ug_B^C$. Note that $\min B\geq 3$ by the assumptions on $i$, hence $\min C\geq 4$.
By induction on $n$, we prove the inequality of the statement, and moreover we prove that the inequality is strict if $\min B\mid n$, or equivalently, if $o_B(n)\in\li$.
Let $b=\baseB n$ and $c = \ug b$.
The case $n<\min B$ is easy since then $o_B(n)=n$ and $\fsc {o_B(n)} i = n-1 = o_C(\ug n - 1)$. Assume henceforth that $n\geq \min B$.

If $b\nmid n$, write $n=ba+r$ with $0<r<b$.
By Lemma~\ref{lemmUp(ba+r)}\ref{lemmUp(ba+r)ba+r}, $\ug n=\ug ba+\ug r$.
By the induction hypothesis, $o_C(\ug r -1) \geq \fsc{o_B(r)}i$, and thus 
\begin{align*}
o_C(\ug n -1) &= o_C(\ug ba + \ug r -1)\\
      &= o_C(\ug ba) + o_C(\ug r -1) \geq o_B(ba) + \fsc{o_B(r)}i = \fsc{o_B(n)}i.
\end{align*}
In the second equality we have used that $\ug ba$ is $C$-critical by Lemma~\ref{lemmbasecriticalpres}\ref{criticalpres}, and $\ug r - 1 < \ug b = c \leq \basep C {\ug ba}$. If $\min B\mid n$, then $\min B\mid r$ and we may replace the above inequality by a strict one.

Now assume that $n= b^2u+bv$ is $B$-critical, and let $d_j = d_j(n,i+1)$ for $0\leq j\leq i+1$. 
Then $o_B(n) = \vartheta(\zeta_B(n))$. We will also denote $\tau \coloneqq \tau(\check\zeta_B(n))$.
We will show that $o_C(d_{i+1} - d_i)\geq \fsc{o_B(n)}i$. The desired result then follows since $\ug n - 1 > d_{i+1} - d_i$ by Lemma~\ref{lemmCanonProp}\ref{lemmCanonPropWitness}.

\begin{Cases}
    \item ($o_B(n)$ contains a non-zero term $\om^{o_B(\tilde v)}$).
    Then $\zeta_B(n) = \alpha_B(n) +\beta_B(n) + \om^{o_B(\tilde v)}$. By Lemma~\ref{lemmThetaStar}, $\zeta_B(n)\notin\FIX$, so $\check\zeta_B(n) = \zeta_B(n)$ and $\tau = \beta_B(n)+\om^{o_B(\tilde v)}$. Then
    \begin{equation}\label{eqIndependenceCase1}
    \fsc{o_B(n)}i = \vartheta(\fs{\check\zeta_B(n)}{\fsc\tau i} + \vartheta^*(\zeta_B(n))) = \vartheta(\alpha_B(n) + \beta_B(n) + \fsc{\om^{o_B(\tilde v)}}i)
    \end{equation}
    We will use the following facts in the subcases below.
    \begin{enumerate}
    \item \label{item1Independence}
    Either $v > 2$ or $u>0$. Then
    \[d_{j+1} = \Chgbases {b}{d_j} (b^2u+bv) = d_j^2 u_j + d_j\cdot \ug v\geq 3d_j,\]
    since in the first case, $\ug v \geq 3$, while in the last case $d_j\geq \min C\geq 4$ and $u_j>0$.
    Therefore $d_{j+1}-d_j\notin C$ is $C$-critical.
    \item \label{item2Independence}
    From $\ug v >0$ we get that $\alpha_C(d_{j+1}-d_j) = \alpha_C^{d_j}(d_{j+1})$. Since $O_C^{d_j}(d_j^2u_j) = O_C^{d_j}(\Chgbases {b}{d_j}b^2u) = O_B(b^2u)$, we have $\alpha_C^{d_j}(d_{j+1}) = \alpha_B(n)$ for all $j$. In particular we see that $d_{j}-d_{j-1}$ is a candidate for $(d_{j+1}-d_{j})_*$ when $j>0$, and in fact it is the maximal candidate, i.e.\ $(d_{j+1}-d_{j})_* = d_{j}-d_{j-1}$. By Lemma~\ref{lemmO*}, $\beta_C(d_{j+1} - d_j) = \ga + 1$ where $o_C(d_j - d_{j-1}) = \vartheta(\alpha_B(n) + \ga)$.
    \end{enumerate}
    
    \begin{Cases}
        \item ($o_B(\tilde v)\in\su$).
        By Lemma~\ref{standardPropOo}, we have $\min B \nmid \tilde v$ and $o_B(\tilde v) = o_B(\tilde v - 1) + 1$.
        Hence also $\min B\nmid v$ and $o_B(v) = o_B(v-1)+1$.
        Then $\fsc{\om^{o_B(\tilde v)}}i = \om^{o_B(\tilde v) - 1}\cdot i = \om^{o_B(\tilde v -1)}\cdot i$ by Lemma~\ref{lemmCFSsuccli}\ref{lemmCFSsucc}.
        We claim that for $0 \leq j < i+1$,
    \begin{equation}\label{eqIndependenceCase1.1}
    \vartheta(\alpha_B(n) + \beta_B(n) + \om^{o_B(\tilde v -1)}\cdot j) \leq o_C(d_{j+1} - d_{j}).
    \end{equation}
    Then for $j=i$, $\fsc{o_B(n)}i \leq o_C(d_{i+1} - d_{i})$. 
    By Lemma~\ref{lemmUp(ba+r)}\ref{lemmUp(ba+r)minBdiv}, $\ug v = \ug(v-1)+1$. Therefore $d_{j+1} - d_j = \Chgbases {b}{d_j} (b^2u+b(v-1)) = \Chgbases {b}{d_j}(n-b)$ holds for all $j<i+1$. Then by Lemma~\ref{lemmCanonProp}\ref{lemmCanonPropWitness}, $\ug (n-b) = \Chgbases {b}{d_0} (n-b) = d_1 - d_0$. 
    
    We now proceed in proving equation (\ref{eqIndependenceCase1.1}) by induction on $j$.
    \begin{Cases}
        \item ($\tilde v=1$).
        This case occurs when either $n=3b$ with $3<\min B$, or $n=b^2u+b$.
        Then $o_C(d_1 - d_0) = o_B(n-b) = \vartheta(\alpha_B(n)+\beta_B(n))$ and by using item \ref{item2Independence} above inductively, $o_C(d_{j+1}-d_j) = \vartheta(\alpha_B(n)+\beta_B(n)+j)$.
        \item ($\tilde v>1$).
        Then the base case $j=0$ follows from 
        \[o_C(d_1 - d_0) = o_B(n - b) = \vartheta(\alpha_B(n) + \beta_B(n) + \om^{o_B(\tilde v - 1)})\geq \vartheta(\alpha_B(n)+\beta_B(n)),\]
        and the induction step from
        \begin{align*}
        o_C(d_{j+1} - d_j) &= \vartheta(\alpha_B(n)+\beta_C(d_{j+1}-d_j) + \om^{o_B(\tilde v -1)})\\
        &\geq \vartheta(\alpha_B(n) + (\beta_B(n) + \om^{o_B(\tilde v -1)}\cdot (j-1) + 1) + \om^{o_B(\tilde v - 1)})\\
        &= \vartheta(\alpha_B(n) + \beta_B(n) + \om^{o_B(\tilde v -1)}\cdot j).
        \end{align*}
        For the first line, one checks that $\widetilde{\ug v - 1} = \ug (\tilde v - 1)$ holds in the case that we are considering. In the second line, we have used Lemma~\ref{propThetaMonEps}\ref{propThetaMon}.
    \end{Cases}

        \item ($o_B(\tilde v)\in\li$).
        Then by Lemma~\ref{standardPropOo}, $\min B\mid \tilde v>0$, in particular $\tilde v = v$.
        \begin{Cases}
        \item ($n = b\cdot\min B$).
        Then $\fsc{\om^{o_B(v)}}i = \fsc{\om^\om} i= \om^i$ and we have that $d_1 = d_0\cdot\min C$.
        Now we will use the following facts.
        \begin{enumerate}
        \item From $i<\min B - 1$ and $\min C = \min B + 1$ we get $i\leq \min C - 3$.
        \item One derives from Lemma~\ref{lemmCanonProp}\ref{lemmCanonPropWitness} that $\ug b \leq d_0$, which implies $\beta_B(n)\leq \beta_C(d_1-d_0)$. Here we also rely on the fact that $\alpha_C(d_1-d_0)=\alpha_B(n)=0$.
        \end{enumerate}
        Thus going back to equation (\ref{eqIndependenceCase1}),
        \begin{align*}
            \fsc{o_B(n)}i = \vartheta(\beta_B(n) + \om^i)
            \leq \vartheta(\beta_C(d_1-d_0) + \om^{\min C - 3}) = o_C(d_1-d_0).
        \end{align*}

        \item ($n\neq b\cdot\min B$).
        Whether or not $\om^{o_B(v)} = o_B(v)$, we have $\fsc{\om^{o_B(v)}}i \leq \om^{\fsc{o_B(v)}i+1}$ by Lemma~\ref{lemmCFSsuccli}\ref{lemmCFSli}, and since $\min B\mid v$, by the induction hypothesis $\fsc{o_B(v)}i+1 \leq o_C(\ug v - 1)$. As in the previous case, we have that
        $\beta_B(n)\leq \beta_C(d_1-d_0)$. Then
        \begin{align*}
        \fsc{o_B(n)}i &\leq \vartheta(\alpha_B(n) + \beta_B(n) + \om^{\fsc{o_B(v)}i+1})\\
        &\leq \vartheta(\alpha_B(n) + \beta_C(d_1-d_0) + \om^{o_C(\ug v - 1)}) = o_C(d_1 - d_0).
        \end{align*}
        \end{Cases}
    \end{Cases}

    \item ($n = 2b$).
    Then $\zeta_B(n) = \beta_B(n)$. We see by Lemma~\ref{lemmO*} and Lemma~\ref{lemmThetaStar} that, whether or not $n_*$ exists, we have $\zeta_B(n)\in\Om\cap\JUMP$. Therefore $\fsc{o_B(n)}i = \vartheta^*(\zeta_B(n))\cdot i$.
    Now note that for $0\leq j < i+1$ we have 
    $d_{j+1} = 2d_j$,
    so $o_C(d_0)\cdot 2^{i} =  o_C(d_{i}) = o_C(d_{i+1}-d_i)$. Therefore it suffices to show that $\vartheta^*(\zeta_B(n))\leq o_C(d_0)$. If $n_*$ does not exist, then by Lemma~\ref{lemmThetaStar}, $\vartheta^*(\zeta_B(n)) = \om = o_C(\min C) \leq o_C(d_0)$. Otherwise $\vartheta^*(\zeta_B(n)) = o_B(n_*) < o_B(b) \leq o_C(d_0)$.

    \item ($n=b^2u$ with $u>0$).
    By Lemma~\ref{lemmO*} and Lemma~\ref{lemmThetaStar}, we see that whether or not $n_*$ exists, $\check\zeta_B(n) = \alpha_B(n)$. So $\tau = \tau(\alpha_B(n))$. We collect the following facts, to be used in the subsequent cases.
    \begin{enumerate}
        \item As in case $1$,
        \[d_{j+1} = \Chgbases b {d_j} n = d_j^2u_j\geq 3d_j,\]
        and $d_{j+1}-d_j\notin C$ is $C$-critical.
        \item 
        We again have $\alpha_C^{d_j}(d_{j+1}) = \alpha_B(n)$ for all $j$, however $\alpha_C(d_{j+1}-d_j) = \alpha_C(d_j^2(u_j-1))$ may in general depend on $j$.
    \end{enumerate}

    There are now two cases to consider.
    \begin{Cases}
        \item ($\tau = \Om$).
        We will prove by induction on $0 \leq j < i+1$ that
        \begin{equation}\label{eqIndependenceCase3.1}
        \fsc{o_B(n)}j \leq \vartheta(\alpha_C({d_{j+1} - d_{j}^2}) + \beta_C(d_{j+1}-d_j) + \om^{o_C(d_j-1)}) = o_C(d_{j+1}-d_j).
        \end{equation}
        The base case $j=0$ follows from
        \[\fsc{o_B(n)}0 = \vartheta^*(\zeta_B(n)) \leq o_B(b) \leq o_C(d_1-d_0),\]
        since $\vartheta^*(\zeta_B(n))$ is either $o_B(n_*)$ or zero.
        For the induction step, assume that $j>0$.
        Then
        \[\fsc{o_B(n)}{j} = \vartheta(\fs{\alpha_B(n)}{\fsc{o_B(n)} {j-1}}) = \vartheta(\alpha_C^{d_j}(d_{j+1})[\fsc{o_B(n)}{j-1}])\]
        We will use Lemma~\ref{propCompareTheta} to show that this last expression is less than $o_C(d_{j+1}-d_j)$. From the induction hypothesis, Proposition~\ref{propCompareFS}\ref{propCompareFSmon} and Corollary~\ref{corAlphaFS}, we get
        \[\fs{\alpha_C^{d_j}(d_{j+1})}{\fsc{o_B(n)}{j-1}} \leq 
        \fs{\alpha_C^{d_j}(d_{j+1})}{o_C(d_j-d_{j-1})} \leq \alpha_C({d_{j+1} - d_{j}^2}) + o_C(d_{j}-d_{j-1}).\]
        Furthermore $o_C(d_j - d_{j-1})< \om^{o_C(d_j - 1)}$. Now it suffices to prove that
        \[\mco{\fs{\alpha_C^{d_j}(d_{j+1})}{\fsc{o_B(n)}{j-1}}} < o_C(d_{j+1} - d_j).\]
        This follows by Proposition~\ref{propCompareFS}, since $\alpha_C^{d_j}(d_{j+1})^* = \alpha_B(n)^*$ is of the form $o_B(v)$ with $v<b$, so $\ug v < d_0$, and $\fsc{o_B(n)}{j-1}\leq o_C(d_j-d_{j-1}) < o_C(d_{j+1}-d_j)$ by the induction hypothesis.

        \item ($\tau < \Om$).
        In this case $\tau = o_B(v)$ for some $v<b$, and since $\alpha_B(n)\notin\su$, $\tau$ must be infinite. We claim that
        \begin{align*}
        \fsc{o_B(n)}i &= \vartheta(\fs{\alpha_B(n)}{\fsc\tau i} + \vartheta^*(\zeta_B(n)))\\
        &= \vartheta(\fs{\alpha_C^{d_0}(d_1)}{\fsc\tau i} + \vartheta^*(\zeta_B(n)))\\
        &\leq \vartheta(\alpha_C(d_1 - d_{0}^2) + \beta_C(d_1-d_0) + \om^{o_C(d_0-1)}) = o_C(d_1-d_0).
        \end{align*}
        Again we will use Proposition~\ref{propCompareTheta} together with Corollary~\ref{corAlphaFS}. There are two cases to consider.
        \begin{Cases}
        \item ($\fsc{\tau} i < \om$).
        By Lemma~\ref{lemmZetaIota}\ref{itXiIotaInf}, $\fsc\tau i = i$. Therefore $0<\fsc{\tau} i < \min C - 1$ satisfies the conditions of Corollary~\ref{corAlphaFS}. So $\fs{\alpha_C^{d_0}(d_1)}{\fsc\tau i}\leq  \alpha_C({d_1- d_{0}^2}) + \fsc{\tau}i$. By the induction hypothesis and $\ug v < d_0$, we get that $\fsc\tau i = \fsc{o_B(v)}{i} \leq o_C(\ug v - 1) < o_C(d_0-1)$. Moreover since $\vartheta^*(\zeta_B(n))$ is either $o_B(n_*)$ or zero, $\vartheta^*(\zeta_B(n)) < o_C(d_0-1)$. Putting everything together,
        \[\fs{\alpha_C^{d_0}(d_1)}{\fsc\tau i} + \vartheta^*(\zeta_B(n)) < \alpha_C({d_1 - d_{0}^2}) + \beta_C(d_1-d_0) + \om^{o_C(d_0-1)}.\]
        As in case $3.2$, all of the coefficients of $\fs{\alpha_C^{d_1}(d_0)}{\fsc\tau i}$, as well as the term $\vartheta^*(\zeta_B(n))$, are less than $o_C(d_1-d_0)$. Since $o_C(d_1-d_0)$ is additively indecomposable,
        \[\mco{\left(\fs{\alpha_C^{d_0}(d_1)}{\fsc\tau i} + \vartheta^*(\zeta_B(n))\right)} < o_C(d_1-d_0).\]

        \item ($\fsc{o_B(v)}i\geq \om$).
        By the induction hypothesis, $\om \leq \fsc\tau i = \fsc{o_B(v)}i \leq o_C(\ug v - 1)$, so 
        \[\fs{\alpha_C^{d_0}(d_1)}{\fsc\tau i} + \vartheta^*(\zeta_B(n))
        \leq \fs{\alpha_C^{d_0}(d_1)}{o_C(\ug v - 1)} + \vartheta^*(\zeta_B(n))).\]
        Obviously $o_C(\ug v -1)$ satisfies the conditions of Corollary~\ref{corAlphaFS}, and we proceed in the same way as in the previous case.
        \end{Cases}
    \end{Cases}
\end{Cases}

Finally, assume that $n = b\in B$.
If $n=\min B$ then $o_B(n) = \om$ and $\fsc\om i = i < \min C-1$, so we have that $\fsc\om i < \ug n-1 = o_C(\ug n - 1)$.
Otherwise, let $d$ be the predecessor of $b$ in $B$ and let $m = b-d$. By preservation, $o_C(\ug n) = o_B(n) = o_C(m)\cdot 2 = o_C(\ug m)\cdot 2$. If $m=d\in B$, then $\ug n = 2\cdot \ug m$ by Lemma~\ref{LemmUpgrTwiceBase}, and we can use the induction hypothesis on $m$ to get
\[
o_C(\ug n - 1)  = o_C(\ug m) + o_C(\ug m - 1)
  > o_B(m)+\fsc{o_B(m)}i =\fsc{o_B(n)}i.
\]
Otherwise $m = b-d$ is $B$-critical. Let $d_j \coloneqq d_j(m,i+1)$ for $j\leq i+1$. Instead of applying the induction hypothesis on $m$, we need the stronger condition that $o_C(d_{i+1} - d_i)\geq \fsc{o_B(m)} i$, which we have shown above. By Lemmas~\ref{lemmUp(ba+r)}\ref{lemmUp(ba+r)b-d} and \ref{lemmCanonProp}\ref{lemmCanonPropBases}, $\ug n = \ug m + \basep C{\ug(n-1)} = \ug m + d_{i+1}$.
Then 
\[
o_C(\ug n - 1) = o_C(\ug m) + o_C(d_{i+1} - 1)
  > o_B(m)+\fsc{o_B(m)}i =\fsc{o_B(n)}i.
\]
This concludes the proof.
\end{proof}

\begin{theorem}\label{theoGoodInd}
Theorem~\ref{theoTerm} is not provable in $\sf KP$.
\end{theorem}

\begin{proof}
Consider the canonical dynamical hierarchy $\mathcal C$ with $C_0=\{3\}$ and $C_{i+1} = (C_i)_{+(i+2)}$. By Lemma~\ref{lemmGoodSuccB(+i)}, each base hierarchy in $\mathcal C$ is a good successor of the previous one. Write $\ug_i$ for $\ug_{C_i}^{C_{i+1}}$ and $o_i$ for $o_{C_i}$.

Let $G \colon\mathbb N\to \mathbb N$ be such that $G(k)$ is the termination time of the Goodstein sequence for $\mathcal C$ starting on $3_{k}$.
Theorem~\ref{theoTerm} implies that $G$ is a total function over $\sf KP$.
We show that $\sf KP$ does not prove its totality, hence it does not prove Theorem~\ref{theoTerm}.
By Theorem~\ref{theoKPInc}, it suffices to show that $G$ is not dominated by $F_\al$, where $\al=\vartheta[\ve_{\Om+1}]$.

For $i\leq G(k)$, let $n_i = \mathbb G^\mathcal C_i(3_{k})$.
Notice that $o_0(n_0) = o_0(3_k) = \vartheta(\Om_k) = \fsc{\al}k$.
We claim that if $i\leq G(k)$ then $\fsi{\fsc\al k}i = \fsi{o_0(n_0)}{i}\leq o_i(n_i)$.
Indeed, the sequence $(o_i(n_i))_{i\leq I}$ satisfies the assumptions of Proposition~\ref{propFunMajor}, since
\[o_{i }(n_{i }) = o_{i+1}(\ug_i n_{i }) > o_ {i+1} (\ug_i n_i-1) \geq  \fsc{o_{i }(n_{i })}{i+1},\]
where the first equality is by Proposition~\ref{propPresO}, the second by Proposition~\ref{propMonO}, and the third by Proposition~\ref{propMajorizeFS}.

It follows that $\fsi{\fsc\al k}i > 0$ whenever $n_i>0$, hence $G(k) \geq F_\al(k)$.
\end{proof}

\begin{remark}\label{remInd}
    We could have also set $C_0 = \{2\}$, $C_{i+1} = (C_i)_{+(i+1)}$, and $n_0=2_{k}+1$. Then by Lemma~\ref{lemmUp(ba+r)}\ref{lemmUp(ba+r)minBdiv}, $\ug n_0 - 1 = \ug (n_0 - 1)$. Therefore $o_1(n_1) = o_0(n_0 - 1) = \vartheta(\Om_{k})$. Now we can repeat the argument of the previous theorem, treating the process as if it starts at $n_1$ instead of $n_0$.

\end{remark}

\begin{remark}
    The dynamical hierarchy $\mathcal C$ appearing in the proof of Theorem~\ref{theoGoodInd} contains base hierarchies which are infinite, except for the first one. As demonstrated in \cite{fernandez2025fractal}, this feature is not essential: using Lemma~\ref{lemmrestrictbasehier}, one can construct a single dynamical hierarchy $\mathcal D=(D_i)_{i\in\N}$ which establishes independence and has the property that every $D_i$ is finite.
\end{remark}

\section{Phase transitions}\label{section_phtr}

In this section, let $B$ be a base hierarchy. We introduce the notation
\[\mathcal{I}(B) \coloneqq \sup_{n\in\N} o_B(n).\]
Note that $\mathcal{I}(B) > \om = o_B(\min B)$. We will investigate further the connection between $\mathcal{I}(B)$ and the structure of $B$.

\begin{lemma}\label{lemmIom^2}
    The following are equivalent.
    \begin{enumerate}
        \item $\mathcal{I}(B) = \om^2$.\label{lemmIom^2_1}
        \item $S_B(b) = 2b$ for every $b\in B$.\label{lemmIom^2_2}
    \end{enumerate}
\end{lemma}

\begin{proof}
    If \ref{lemmIom^2_2} holds then we see that $o_B(b_k) = \om\cdot 2^{k-1}$, where $b_k$ is the $k$-th base of $B$. 

    On the other hand if \ref{lemmIom^2_2} does not hold, then there exists a $B$-critical element $2b$ which is not a base. By the following case distinction, $o_B(2b)\geq \om^2$, so that \ref{lemmIom^2_1} does not hold.
    \begin{Cases}
    \item ($b=\min B = 2$).
    Then $o_B(2b) = \vartheta(\Om^\Om)$.
    \item ($b\neq \min B = 2$).
    Then $o_B(2b) = \vartheta(\beta_B(2b) + \om^\om)$.
    \item ($\min B > 2$).
    Then $o_B(2b) = \vartheta(\beta_B(2b))$, and since $o_B(b)\geq\om$ we have $\beta_B(2b)\geq 2$.
    \end{Cases}
\end{proof}

\begin{lemma}\label{lemmIom_2}
    The following are equivalent.
    \begin{enumerate}
        \item $\om^2 <\mathcal{I}(B) \leq \om^\om$.\label{lemmIom_2_1}
        \item $S_B(b) \leq 3b\leq b\cdot \min B$ for every $b\in B$, and $S_B(b) = 3b$ for some $b\in B$.\label{lemmIom_2_2}
    \end{enumerate}
    Moreover, $\mathcal I(B) = \om^\om$ exactly when there are infinitely many $b\in B$ with $S_B(b)=3b$.
\end{lemma}

\begin{proof}
    Suppose \ref{lemmIom_2_2} holds. Consider the least $b\in B$ such that $S_B(b) = 3b$. Then $2b$ is $B$-critical, and $(2b)_*$ does not exist. By Lemma~\ref{lemmO*},
    $o_B(2b) = \vartheta(2) = \om^2$, so $\om^2 <\mathcal{I}(B)$. We now show by induction on $n$ that $o_B(n)<\om^\om$, which implies $\mathcal{I}(B)\leq \om^\om$. We assume that $n=2b$ is not a base since the other cases are trivial. Then
    $o_B(n) = \vartheta(\beta_B(n))$. By Lemma~\ref{lemmO*}, $\beta_B(n)$ is either two or of the form $\zeta+1$, where $o_B(n_*) = \vartheta(\zeta) < \om^\om$ by the induction hypothesis. It follows from Lemma~\ref{propThetaMonEps}\ref{propThetaMon} that $\zeta<\om$, so $\beta_B(n)<\om$ and $o_B(n)<\om^\om$.

    To see the final assertion of the statement, note that if $S_B(b) = 3b$, then $(2b)_*$ is the previous element $2d$ with $d\in B$ and $S_B(d)=3d$. If there are infinitely many such $b\in B$ with $S_B(b)=3b$, then $\zeta$ will keep growing by one, and $\mathcal{I}(B) = \sup_n(\vartheta(n)) = \om^\om$. Conversely, if there is some $b\in B$ such that all its successors $d\in B$ (including $b$ itself) satisfy $S_B(d)=2d$, then $o_B(b)<\om^\ell$, and for all successors $d\in B$, $o_B(d) = o_B(b)\cdot 2^k < \om^{\ell+1}$.

    Conversely if \ref{lemmIom_2_2} does not hold, then we have three cases.
    \begin{Cases}
    \item ($b\cdot \min B < S_B(b)$ for some $b\in B$).
        Then $b\cdot\min B$ is $B$-critical and not a base, and we consider the following subcases.
        \begin{Cases}
        \item ($b=\min B = 2$).
        Then $o_B(b\cdot \min B) = \vartheta(\Om^\Om)>\om^\om$. 
        \item ($b=\min B > 2$).
        Then $o_B(b\cdot\min B) = \vartheta(\Om) > \om^\om$.
        \item $(b>\min B)$.
        Then $o_B(b\cdot\min B) = \vartheta(\beta_B(b\cdot\min B) + \om^\om)>\om^\om$.
        \end{Cases}
        
        \item ($3b <S_B(b) \leq b\cdot \min B$ for some $b\in B$).
        Then $3b$ is $B$-critical and not a base, and $o_B(3b) = \vartheta(\beta_B(3b) + \om)\geq \om^\om$.

    \item ($S_B(b) = 2b$ for every $b\in B$). By Lemma~\ref{lemmIom^2}, $\mathcal{I}(B) = \om^2$.

    \end{Cases}
    In any case we see that \ref{lemmIom_2_1} does not hold.
\end{proof}

\begin{theorem}\label{TheoProvRCAo}
    Let $k\geq 0$ be given.
    $\sf RCA_0$ proves the following. Let $\mathcal B = (B_i)_{i\in\N}$ be any dynamical hierarchy. Suppose that $B_0$ is such that $S_{B_0}(b) \leq \min(3b, b\cdot\min {B_0})$ for every $b\in B_0$, and $S_{B_0}(b)=3b$ for $k$ many $b\in B_0$. Then for every $n$ there is some $i$ such that $\goodp ni {\mathcal B} = 0$.
\end{theorem}
\begin{proof}
    First we fix a $\Delta^0_0$-formula $\mathrm{Hist}(B,C, s,n)$ of second order arithmetic which expresses that $s$ is a (coded) calculation of the upgrade together with the operators $\Chgbases bc$, up to $n$, where $b\leq \baseB n$ and $c$ ranges up to the witness of $\ug n$. The parameters $B, C$ will be omitted if they are clear from the context. More precisely, $s$ should contain:
    
    \begin{enumerate}
        \item The values $\ug m$ for $m\leq n$, together with their witness if $m\geq \min B$.
        \item The values $\Chgbases bc m$ for $m\leq n$, $b\leq \baseB n$ and $c\leq \basep C{\ug n}$.
    \end{enumerate}
    
    All of the data can be obtained from $s$ by primitive recursion. In particular we will write $s_n$ for the value of $\ug n$ inside $s$.
    Then the upgrade has the following $\Sigma^0_1$ definition.
    \[\ug n = m \iff \exists s (\mathrm{Hist}(s,n)\land s_n = m)\]
    Moreover, totality of the upgrade is expressed by the $\Pi^0_2$-formula $\forall n\exists s \mathrm{Hist}(s,n)$. We use this formula to formalize the expression ``$C$ is a good successor of $B$", which in turn is used to formalize ``$(B_i)_i$ is a dynamical hierarchy". The hypothesis on $B_0$ in the statement is also easily formalized. 
    
    We now argue that all of the properties of the upgrade operator, which are proven in \cite{fernandez2025fractal}, are provable in $\sf RCA_0$. We consider only the ones needed for proving termination and we add for convenience to every statement the good successor hypothesis, so that $\ug$ is total.  
    Apart from induction, these proofs use only logical rules. Therefore it suffices to check that every formula we apply induction on is at most $\Sigma^0_1$. For example, for monotonicity of the upgrade we apply induction on the $\Delta^0_0$-formula
    \[\mathrm{Hist(s,n)} \rightarrow (\forall m < n: s_m < s_n).\]
    One checks now that we can use the same trick every time induction is needed.
    We prove all other items of Lemma~\ref{lemmMonUg}, as well as Lemmas \ref{lemmStructureChgbases}, \ref{lemmbasecriticalpres} and \ref{lemmalternativedefupgr}. Item \ref{lemmUp(ba+r)witnessbase} of Lemma~\ref{lemmUp(ba+r)} requires no induction. For item \ref{lemmUp(ba+r)ba+r}, one proves first the equality $\Chgbases bc (ba+r) = \Chgbases bc ba + \ug r$ by induction on $a$, whenever $c\leq \basep C {\ug(ba+r)}$. We can express this equation again as a $\Delta^0_0$-formula using $\mathrm{Hist}$. After that we prove \ref{lemmUp(ba+r)ba+r}. Moreover \ref{lemmUp(ba+r)minBdiv} follows easily from \ref{lemmUp(ba+r)ba+r}. Finally, \ref{lemmUp(ba+r)b-d} requires no induction but uses the Lemmas we have just discussed.

    To be able to treat ordinals below $\vartheta[\ve_{\Om+1}]$ within $\sf RCA_0$, we use the notation system given in \cite{FernWierCiE2020}. Namely, terms representing ordinals 
    are built up from the constant $0$ and the functions $x\mapsto \vartheta(x)$, $(x,y)\mapsto x+y$ and $(x,y)\mapsto \Om^xy$. It is shown in \cite{FernWierCiE2020} that this system is complete, i.e.\ for every $\xi < \vartheta[\ve_{\Om+1}]$, there exists a term which represents $\xi$. We can then define the following primitive recursive functions and relations, simultaneously by recursion on terms:
    
    \begin{enumerate}[label = (\arabic*)]
        \item $x\mapsto \bar x$, the normal form of $x$. Every subterm of $\bar x$ should be in $\Omega$-normal form.
        \item $x\prec y$, stating that the ordinal of $x$ is less than the ordinal of $y$.
        \item $x\mapsto \mco x$, the maximal coefficient of $x$.
    \end{enumerate}
    For example, $\vartheta(x)\prec \vartheta(y)$ if either $x\prec y$ and $\mco x \prec \vartheta(y)$, or $\mco y \succcurlyeq x$. Deciding whether $\mco y = x$ comes down to reducing $\mco y$ and $x$ to normal form and checking if the terms match. As another example, to transform $x+y$ into normal form, write $\bar x = x_1 + \dots + x_m$ and $\bar y = y_1+\dots + y_n$ with $x_i,y_i$ additively indecomposable, and output $x_1+\dots +x_k + y_1 + \dots + y_n$, where $k$ is maximal such that $x_k\succcurlyeq y_1$.

    Moreover we have the following primitive recursive functions and relations.
    \begin{enumerate}[label = (\arabic*)]
    \item $\mathrm{Term}(x)$, stating that $x$ is the code of some ordinal.
    \item
    $\mathrm{SubTerm}(x,y)$, stating that $x$ is a subterm of $y$.
    \item $\lvert x \rvert$, the size of the term $x$.
    \end{enumerate}
    We then prove in $\sf RCA_0$ that $\mco x\prec \vartheta(y)$ whenever $x$ is a subterm of $y$, by induction on $\lvert x\rvert + \lvert y \rvert$. In particular item \ref{propThetaMco} of Lemma~\ref{propThetaMonEps} follows in $\sf RCA_0$. The other items of Lemma~\ref{propThetaMonEps}, together with Proposition~\ref{propCompareTheta}, follow from this and the definition of $\prec$ (except for the surjectivity of $\vartheta$, which we do not need for termination).

    We now work towards a definition of the ordinal assignments in $\sf RCA_0$.
    The function $O_f^b$ in Definition \ref{defO_f^b} can be defined by primitive recursion in $f$ and $b$. Hence it has a $\Delta^0_0$ definition with $f$ and $b$ as parameters. Lemmas \ref{lemmO_fMon} and \ref{lemmO_fMultOm}, which are proven by induction \cite{fernandez2025fractal}, therefore hold in $\sf RCA_0$.

    Note that, with our notation system, we can introduce the primitive recursive function $x\mapsto \om^x$ by checking whether $\bar x$ is of the form $\vartheta(y) + \vartheta(0)+\dots + \vartheta(0)$ for $y\succcurlyeq \Om$ (Lemma~\ref{lemmVarthetaLessOm}).
    Then for every $n$, we can compute (the code of) $o_B(n)$ by primitive recursion in $B$. Indeed, we can derive the value of $\beta_B(n)$ from Lemma~\ref{lemmO*}. So we fix a $\Delta^0_0$-formula $\mathrm{Trace}(B, w,n)$, expressing that $w$ is a correct calculation which contains
    
    \begin{enumerate}
        \item The values $o_B(m)$ for $m\leq n$.
        \item The values $O_B^b(m)$ for $m\leq n$ and $b\leq \baseB n$.
        \item The values $\alpha_B^b(m)$ for $m\leq n$ and $b\leq \baseB n$.
    \end{enumerate}
    In other words $\mathrm{Trace}$ is the analogue of $\mathrm{Hist}$, but verifies calculations for the ordinal assignment.
    Now the proofs of Lemma \ref{standardPropOo}, Proposition~\ref{propMonO}, Lemmas \ref{lemmInductiveStepAlpha} to \ref{lemmO*}, and Proposition~\ref{propPresO}, are formalized in $\sf RCA_0$, by using the formula $\mathrm{Trace}$ whenever induction is needed. We have also covered all properties of $\ug$ and $\vartheta$ which are used in these proofs. 

    We can prove Lemma~\ref{lemmIom^2} and Lemma~\ref{lemmIom_2} in $\sf RCA_0$, again since the proofs require only induction on $\Sigma^0_1$-formulas. By the conditions on $B_0$ given in the statement, $\sf RCA_0$ proves that $o_{B_0}(n) < \om^{k+3}$ for every $n$. 

    Finally, we prove in $\sf RCA_0$ the statement $\forall n \exists i (\goodp ni {\mathcal B} = 0)$. Let $n$ be arbitrary.
    We use for the function $i \mapsto \goodp ni{\mathcal B}$ a $\Sigma^0_1$-definition which uses our definition for the upgrade stated earlier.
    We define the mapping $f$ which sends $i\in\N$ to (the code of) $o_i(\goodp ni {\mathcal B})$. Then $f$ also has a $\Sigma^0_1$-definition which uses the definitions of $\goodp ni{\mathcal B}$ and $o_{B_i}$ for all $i$. In particular note that the definition of $f$ has as a parameter the sequence $(B_i)_i$.
    By Propositions \ref{propMonO} and \ref{propPresO}, $\sf RCA_0$ proves that $f(i+1)\prec f(i)$ whenever $i$ is such that $\goodp ni {\mathcal B}\neq 0$. 

    Now consider the formula $\varphi(\alpha) \equiv \exists i (f(i)\preccurlyeq \al)$ which has parameters $\al$ and $(B_i)_i$. It suffices to find a coded ordinal $\alpha$ which is minimal with respect to $\preccurlyeq$. Then $\alpha = f(i)$ for some $i$, and $\goodp ni {\mathcal B} = 0$ since otherwise $f(i+1)\prec f(i)$.
    For this we use the principle of transfinite induction on the $\Pi^0_1$-formula $\psi \equiv \lnot\varphi$ \cite{sommer1995transfinite}. Namely, for every $\ga\prec\om^\om$,
    \[\mathsf{RCA_0} \vdash \forall \alpha((\forall \beta \prec\alpha\;\psi(\beta) ) \rightarrow \psi(\al) ) \rightarrow \forall\al \prec \ga \;\psi(\al).\]
    If we take $\ga = \om^{k+1}+1$, then $\sf RCA_0$ proves the negation of the consequent, hence the negation of the antecedent, namely the existence of a minimal $\alpha$ such that $\varphi(\al)$.
\end{proof}

On the contrary, by repeating the proof of Theorem~\ref{theoGoodInd}, we obtain the following.

\begin{theorem}\label{TheoIndRCAo}
    Let $B_0$ be any base hierarchy which does not satisfy the conditions stated in the previous theorem. Define $B_{i+1} = (B_i)_{+(i+1)}$ for every $i$. Then $\sf RCA_0$ does not prove that for all $n$, there is some $i$ such that $\goodp ni {\mathcal B} = 0$.
\end{theorem}
\begin{proof}
    Let $\al = \om^\om$. By Lemma~\ref{lemmIom^2} and Lemma~\ref{lemmIom_2}, we know that $\mathcal{I}(B_0)\geq \al$. Then for every $k$, we can find $n_k\in\N$ such that $o_0(n_k) \geq \fsc\al k$. We define $G(k)$ as the termination time of the Goodstein sequence starting at $n_k$, and by Theorem~\ref{theoKPInc} it suffices to show that $G$ is not dominated by $F_\al$.
    We can now repeat the argument of Theorem~\ref{theoGoodInd} and use the trick explained in Remark~\ref{remInd}.
\end{proof}

We can also show that $\sf RCA_0$ does not prove the statement of Theorem~\ref{TheoProvRCAo} while quantifying over $k$, in the following way. We consider the following function $G:\N\rightarrow \N$. Given $k$, we consider some base hierarchy $B_{0,k}$ satisfying 
    \begin{enumerate}
    \item
    $S_{B_{0,k}}(b)\leq 3b\leq b\cdot \min B_{0,k}$ for every $b\in B_{0,k}$.
    \item 
    $S_{B_{0,k}}(b) = 3b$ for $k$ many $b\in B_{0,k}$.
    \end{enumerate}
    Let $B_{i+1,k} = (B_{i,k})_{+(i+1)}$ for every $i$. We take some $n$ such that $o_{B_{0,k}}(n) > \om^k$, and set $G(k)$ equal to the termination time of the Goodstein sequence at $n$. If $\sf RCA_0$ was able to prove the statement of Theorem~\ref{TheoProvRCAo} uniformly over $k$, then in particular it would prove the totality of $G$. But $G(k)\geq F_\al(k)$ for all $k$, where $\al=\om^\om$.

\begin{lemma}\label{lemmIom_3}
    The following are equivalent.
    \begin{enumerate}
        \item $\om^\om < \mathcal{I}(B) < \om^{\om^\om}$.\label{lemmIom_3_1}
        \item $S_B(b) \leq b\cdot \min B$ for every $b\in B$, and $S_B(b) > 3b$ for some $b\in B$.\label{lemmIom_3_2}
    \end{enumerate}
\end{lemma}

\begin{proof}
    If \ref{lemmIom_3_2} holds, consider some $b$ such that $3b<S_B(b)\leq b\cdot \min B$. Then
    $o_B(3b) = \vartheta(\beta_B(3b) + \om) \geq \om^\om$, so $\om^\om <\mathcal{I}(B)$. We show by induction on $n$ that $o_B(n)<\om^{\om^{\min B}}$. Assume $n\notin B$ is $B$-critical, the other cases are trivial. Then $n$ is of the form $bv$ with $v<\min B$, and $o_B(n) = \vartheta(\beta_B(n) + \om^{o_B(v-2)}) \leq \vartheta(\beta_B(n) + \om^{\min B-2})$. By applying Lemma~\ref{lemmO*} and possibly the induction hypothesis on $n_*$, $\beta_B(n) < \om^{\min B}$. It follows that $o_B(n)<\om^{\om^{\min B}}$.

    If \ref{lemmIom_3_2} does not hold, then either $S_B(b)\leq b\cdot \min B$ and $S_B(b)\leq 3b$ for every $b\in B$, in which case the previous two lemmas imply that \ref{lemmIom_3_1} does not hold, or $S_B(b)>b\cdot \min B$ for some $b\in B$. In the latter case we have that $b\cdot \min B$ is $B$-critical and not a base itself. By the same case distinction as in the first case of Lemma~\ref{lemmIom_2}, one sees that $o_B(b\cdot \min B)\geq \om^{\om^\om}$.
\end{proof}

By repeating the proofs of Theorem~\ref{TheoProvRCAo} and Theorem~\ref{TheoIndRCAo}, we obtain the following results.

\begin{theorem}\label{TheoProvRCAoPi2}
    Let $k\geq 2$ be given.
    $\mathsf{RCA_0} + (\Sigma^0_2)-\mathrm{IND}$ proves the following. Let $\mathcal B = (B_i)_{i\in\N}$ be any dynamical hierarchy. Suppose that $B_0$ is such that $S_{B_0}(b) \leq b\cdot\min {B_0}$ for every $b\in B_0$, and $\min B_0\leq k$. Then for every $n$ there is some $i$ such that $\goodp ni {\mathcal B} = 0$.
    $\mathsf{RCA_0} + (\Sigma^0_2)-\mathrm{IND}$ does not prove this statement quantified over $k$.
\end{theorem}

\begin{theorem}\label{TheoIndRCAoPi2}
    Let $B_0$ be any base hierarchy which does not satisfy the conditions stated in Theorem~\ref{TheoProvRCAoPi2}. Define $B_{i+1} = (B_i)_{+(i+1)}$ for every $i$. Then $\mathsf{RCA_0} + (\Sigma^0_2)-\mathrm{IND}$ does not prove that for all $n$, there is some $i$ such that $\goodp ni{\mathcal B} = 0$.
\end{theorem}

\begin{lemma}\label{lemmIom_n}
    For $n>0$, the following are equivalent.
    \begin{enumerate}
    \item $\om_{2n+1} < \mathcal{I}(B) < \om_{2n+3}$.\label{lemmIom_n_1}
    \item $S_B(b)\leq b^2$ for every $b\in B$, and there exists a chain of bases $b_0< \dots < b_n$ with maximal length $n$ such that $S_B(b_{i}) > b_i\cdots b_0$ holds for $0\leq i \leq n$.\label{lemmIom_n_2}
    \end{enumerate}
\end{lemma}

\begin{proof}
    First, we argue that if \ref{lemmIom_n_2} holds, then the following greedy algorithm will give us a chain of bases of maximal length satisfying the conditions of \ref{lemmIom_n_2}:
    \begin{statement}
        Let $d_0 = \min B$, and for $i>0$ define $d_i\in B$ recursively as the least base above $d_{i-1}$ which satisfies $S_B(d_i) > d_i\cdots d_0$, if such a $d_i$ exists.
    \end{statement}
    Indeed, let $b_0<\cdots <b_n$ be a chain given by \ref{lemmIom_n_2}. Then we see by induction on $i$ that $d_i\leq b_i$ since $d_i$ is chosen to be minimal, and hence if $b_{i+1}$ exists then so does $d_{i+1}$. So $d_i$ is defined for $i\leq n$, and by the maximality of $n$, $d_{n+1}$ does not exist.

    Now let $d_0<d_1<\cdots$ be defined by our algorithm.
    We will show the following by induction on $m\geq d_0$:
    
    \begin{statement}
        If $d_i\cdots d_0\leq m < d_{i+1}\cdots d_0$, 
        then \[\om_{2i+1} \leq o_B(m) < \vartheta(\om^{o_B(d_i\cdots d_0 - 1)}\cdot\om) < \om_{2i+2}(\min B).\]
        Here $\om_k(\ell)$ denotes the tower of $k$ many $\omega$, with a power of $\ell$ added to the top $\omega$.
        If $d_{n+1}$ is undefined, then the above inequality holds for all $m\geq d_n\cdots d_0$.
    \end{statement}
    Then if \ref{lemmIom_n_2} holds, we can assume without loss of generality that $b_i=d_i$ for all $i$, and \ref{lemmIom_n_1} follows from the last line of the above claim.

    The base case $m = d_0=\min B$, as well as the cases where $m$ is a base or a non-$B$-critical element, are obvious. Consider the case $m = d_{i+1}\cdots d_0$. Then
    \[o_B(m) = \vartheta(\beta_B(m) + \om^{o_B(d_i\cdots d_0)}).\]
    By Lemma~\ref{lemmO*}, $\beta_B(m)$ is either two or $\zeta + 1$, where $o_B(m_*) = \vartheta(\zeta)$. Possibly applying the induction hypothesis on $m_*$, we get 
    \[\beta_B(m) < \om^{o_B(d_i\cdots d_0 - 1)}\cdot \om \leq \om^{o_B(d_i\cdots d_0)}.\]
    Therefore $o_B(m) = \vartheta(\om^{o_B(d_i\cdots d_0)})$, and the claim follows by the induction hypothesis together with $o_B(d_i\cdots d_0)<o_B(d_{i+1}\cdots d_0 - 1)$.

    Now assume that $m\notin B$ is $B$-critical, $d_i\cdots d_0 < m$, and in case $d_{i+1}$ exists, $m < d_{i+1}\cdots d_0$. The lower bound in the claim then follows by applying the induction hypothesis to $d_i\cdots d_0$. For the upper bound, note that as before, $\beta_B(m) < \om^{o_B(d_i\cdots d_0-1)}\cdot\om$.

    Let $m=dv$ with $d=\baseB m$ and $v<d$. Then $o_B(m) = \vartheta(\beta_B(m) + \om^{o_B(\tilde v)})$, where $\tilde v$ is either $v$ or $v-2$. To prove $o_B(m)<\vartheta(\om^{o_B(d_i\cdots d_0-1)}\cdot \om)$, it suffices that $\tilde v \leq d_i\cdots d_0-1$. 
    Note that $d_i\leq d$, and $d\leq d_{i+1}$ if $d_{i+1}$ exists. We consider several cases.
    \begin{Cases}
        \item $(d=d_i)$.
        Then $\tilde v\leq v < d = d_i \leq d_i\cdots d_0$, hence $\tilde v \leq d_i\cdots d_0 - 1$.

        \item $(d_i<d<d_{i+1})$.
        By the definition of $d_{i+1}$, $m<S_B(d)\leq d\cdot d_i\cdots d_0$, so $v<d_i\cdots d_0$ and
        $\tilde v \leq v \leq d_i\cdots d_0-1$.

        \item $(d=d_{i+1})$.
        From $m < d_{i+1}\cdots d_0$ we get $v<d_i\cdots d_0$, so $\tilde v \leq v\leq d_i\cdots d_0 - 1$.
    \end{Cases}

    Now for the converse implication, assume that \ref{lemmIom_n_2} does not hold.
    \begin{Cases}
        \item ($S_B(b) > b^2$ for some $b\in B$).
        Then we have a $B$-critical element $b^2\notin B$, and by an easy case distinction one checks that
        $o_B(b^2) \geq \vartheta(\Om) > \om_{2n+3}$.
        In the cases which follow, we assume that $S_B(b)\leq b^2$ for all $b\in B$. We refer to the algorithm which gives the bases $d_i$ at the beginning of the proof as the greedy algorithm.
        
        \item (The greedy algorithm gives a finite chain $d_0<\dots<d_m$ with $0<m\neq n$). 
        Then by our claim, \ref{lemmIom_n_1} holds for $m$ instead of $n$.

        \item (The greedy algorithm gives an infinite chain $d_0<d_1<\cdots$).
        Then by our claim, $\mathcal{I}(B) = \ve_0$, in particular \ref{lemmIom_n_1} does not hold.

        \item (The greedy algorithm gives $d_0$).
        Since $d_1$ does not exist, for all $b\in B$ we have $S_B(b)\leq b\cdot b_0 = b\cdot\min B$. By the previous three lemmas, $\mathcal{I}(B)<\om_3\leq \om_{2n+1}$.
    \end{Cases}
\end{proof}

\begin{theorem}\label{TheoProvRCAoPin}
    Let $n>0$ and $k\geq 2$ be given. $\mathsf {RCA_0} + (\Sigma^0_{2n+2})-\mathrm{IND}$ proves the following. Let $\mathcal B = (B_i)_{i\in\N}$ be any dynamical hierarchy. Suppose that $B_0$ is such that 
    \begin{itemize}
    \item
    $S_{B_0}(b) \leq b^2$ for every $b\in B_0$.
    \item 
    There is no chain of bases $b_0<\cdots < b_{n+1}$ satisfying $S_{B_0}(b_i) > b_i\cdots b_0$ for $0\leq i \leq n+1$.
    \item
    $\min B_0\leq k$.
    \end{itemize}
    Then for every $m$ there is some $i$ such that $\goodp mi {\mathcal B} = 0$. $\mathsf {RCA_0} + (\Sigma^0_{2n+1})-\mathrm{IND}$ does not prove this statement quantified over $k$.
\end{theorem}

\begin{theorem}
    Let $n>0$ and let $B_0$ be any base hierarchy which does not satisfy the conditions stated in Theorem~\ref{TheoProvRCAoPin}. Define $B_{i+1} = (B_i)_{+(i+1)}$ for every $i$. Then $\mathsf{RCA_0} + (\Sigma^0_{2n+2})-\mathrm{IND}$ does not prove that for all $m$, there is some $i$ such that $\goodp mi {\mathcal B} = 0$.
\end{theorem}

\begin{lemma}\label{lemmIeps0}
    The following are equivalent.
    \begin{enumerate}
        \item $\mathcal{I}(B) < \ve_0$.\label{lemmIeps0_1}
        \item $S_B(b)\leq \min(b^2, bc)$ for every $b\in B$, where $c\geq 2$ is independent of $b$.\label{lemmIeps0_2}
    \end{enumerate}
\end{lemma}

\begin{proof}
    If \ref{lemmIeps0_2} holds, then every chain of bases $b_0<b_1<\cdots$ satisfying $S_B(b_i)>b_i\cdots b_0$ is necessarily of length less than some finite number $n_c$ (one can take $n_c = \log_2(c)+1$). By the previous lemmas, \ref{lemmIeps0_1} holds.

    On the other hand if \ref{lemmIeps0_2} does not hold, then we have two cases.
    \begin{Cases}
    \item ($S_B(b)>b^2$ for some $b\in B$).
    In this case $o_B(b^2)\geq \vartheta(\Om)= \ve_0$.
    \item ($S_B(b)\leq b^2$ for every $b\in B$, but for every constant $c$ there is a $b$ such that $S_B(b)>bc$).
    Then we easily construct an infinite chain $b_0<b_1<\cdots$ satisfying $S_B(b_i)>b_i\cdots b_0$ for all $i$, and from the claim in the proof of the previous lemma we get $\mathcal{I}(B)=\ve_0$. 
     \end{Cases}
\end{proof}

\begin{theorem}\label{TheoProvACAo}
    Let $c\geq 2$ be given. Then $\mathsf {ACA_0}$ proves the following. Let $\mathcal B = (B_i)_{i\in\N}$ be any dynamical hierarchy. Suppose that $B_0$ is such that $S_{B_0}(b) \leq \min(b^2,bc)$ for every $b\in B_0$. 
    Then for every $n$ there is some $i$ such that $\goodp ni {\mathcal B} = 0$.
    $\mathsf {ACA_0}$ does not prove this statement quantified over $c$.
\end{theorem}

\begin{theorem}\label{TheoIndACAo}
    Let $B_0$ be any base hierarchy which does not satisfy the conditions stated in Theorem~\ref{TheoProvACAo}. Define $B_{i+1} = (B_i)_{+(i+1)}$ for every $i$. Then $\sf ACA_0$ does not prove that for all $n$, there is some $i$ such that $\goodp ni {\mathcal B} = 0$.
\end{theorem}

\begin{lemma}\label{lemmIgam0}
    The following are equivalent.
    \begin{enumerate}
        \item $\mathcal{I}(B) < \Gamma_0$.\label{lemmIgam0_1}
        \item One of the following holds:
        \begin{itemize}
        \item $\min B = 2$ and $S_B(b)\leq b^2$ for every $b\in B$.
        \item $\min B > 2$ and $S_B(b)\leq \min(b^3, b^2c)$ for every $b\in B$, where $c\geq 2$ is independent of $b$.
        \end{itemize} \label{lemmIgam0_2}
    \end{enumerate}
\end{lemma}

\begin{proof}
    Assume \ref{lemmIgam0_2}. In the case where $\min B = 2$, we have that $\mathcal{I}(B)\leq \ve_0 < \Gamma_0$ by the claim we proved in Lemma~\ref{lemmIom_n}. Suppose further that the second item of \ref{lemmIgam0_2} holds. Note that for all $n\in\N$, either $\alpha_B(n) = 0$ or $\alpha_B(n)=\Om$. Therefore we can use Lemma~\ref{lemmInductiveStepAlpha} inductively to see that $o_B(n) < \vartheta(\Om\cdot o_B(c))$ as well as $o_B(n)<\Gamma_0$. Then using $o_B(c)<\Gamma_0$ and Proposition~\ref{propCompareTheta}, $\mathcal I(B)\leq \vartheta(\Om\cdot o_B(c)) < \vartheta(\Om^2)=\Gamma_0$.

    Conversely, suppose that \ref{lemmIgam0_2} does not hold. 
    \begin{Cases}
        \item ($\min B=2$).
        Then there is a $B$-critical element $b^2$ with $o_B(b^2) \geq \vartheta(\Om^\om) > \Gamma_0$.

        \item ($\min B > 2$).
        \begin{Cases}
            \item ($S_B(b)>b^3$ for some $b\in B$).
            Then $o_B(b^3) \geq \vartheta(\Om^3) > \Gamma_0$.

            \item ($S_B(b)\leq b^3$ for every $b\in B$).
            Then for every given $c$ we can find $b\in B$ such that $S_B(b)>b^2c$. So we can construct an infinite chain of bases $b_0<b_1<\cdots$ such that $b_0=\min B$ and for $i>0$, $b_i$ is the least base above $b_{i-1}$ satisfying $S_B(b_i)>b_i^2\cdots b_0^2$. It is now easy to show by induction that $o_B(b_i^2\cdots b_0^2)\geq \fsc{\Gamma_0}i$ for every $i$. Indeed, $\fsc{\Gamma_0}{0} = \vartheta^*(\Om^2) = 0\leq o_B(b_0^2)$, and
            \[\fsc{\Gamma_0}{i+1} = \vartheta(\fs{\Om^2}{\fsc{\Gamma_0}{i}}) = \vartheta(\Om\cdot \fsc{\Gamma_0}{i})\leq \vartheta(\Om\cdot o_B(b_i^2\cdots b_0^2))\leq o_B(b_{i+1}^2\cdots b_0^2).\]
            Since $\sup_{i\in\N}\fsc{\Gamma_0}{i} = \Gamma_0$, we get $\mathcal{I}(B)\geq \Gamma_0$. In fact it is easily seen that $\mathcal{I}(B) = \Gamma_0$ in this case. Nevertheless, \ref{lemmIgam0_1} does not hold.
        \end{Cases}
    \end{Cases}
\end{proof}

\begin{theorem}\label{TheoProvATRo}
    Let $c\geq 2$ be given. Then $\mathsf {ATR_0}$ proves the following. Let $\mathcal B = (B_i)_{i\in\N}$ be any dynamical hierarchy. Suppose that $B_0$ is such that one of the following holds:
    \begin{itemize}
    \item
    $\min B_0 = 2$ and $S_{B_0}(b)\leq b^2$ for every $b\in B_0$.
    \item
    $\min B_0 > 2$ and $S_{B_0}(b) \leq \min(b^3,bc)$ for every $b\in B_0$. 
    \end{itemize}
    Then for every $n$ there is some $i$ such that $\goodp ni {\mathcal B} = 0$.
    $\mathsf {ATR_0}$ does not prove this statement quantified over $c$.
\end{theorem}

\begin{theorem}\label{TheoIndATRo}
    Let $B_0$ be any base hierarchy which does not satisfy the conditions stated in Theorem~\ref{TheoProvATRo}. Define $B_{i+1} = (B_i)_{+(i+1)}$ for every $i$. Then $\sf ATR_0$ does not prove that for all $n$, there is some $i$ such that $\goodp ni {\mathcal B} = 0$.
\end{theorem}

\begin{lemma}\label{lemmIBH}
    For $n\geq 2$, the following are equivalent.
    \begin{enumerate}
        \item $\mathcal{I}(B) \leq \vartheta(\Om_n)$.\label{lemmIBH_1}
        \item $S_B(b)\leq b_n$ for every $b\in B$.\label{lemmIBH_2}
    \end{enumerate}
\end{lemma}

\begin{proof}
    If \ref{lemmIBH_2} holds, then for every $m\in\N$, $\alpha_B(m) < \Om_n$. By using Lemma~\ref{lemmInductiveStepAlpha} inductively, $o_B(m) < \vartheta(\Om_n)$ for every $m\in\N$.

    On the other hand if \ref{lemmIBH_2} does not hold, then we find a $B$-critical element $b_n$ with $\alpha_B(b_n) = \Om_n$, hence $o_B(b_n) \geq \vartheta(\Om_n)$.
\end{proof}

\begin{theorem}\label{TheoProvKPPin}
    If $n\geq 2$, then $\mathsf {KP^-\om} + (\Pi_n)-\mathrm{IND}$ proves the following. Let $\mathcal B = (B_i)_{i\in\N}$ be any dynamical hierarchy. Suppose that $B_0$ is such that $S_{B_0}(b)\leq b_n$ for every $b\in B_0$.
    Then for every $m$ there is some $i$ such that $\goodp mi {\mathcal B} = 0$.
\end{theorem}

\begin{theorem}\label{TheoIndKPPin}
    Let $n\geq 2$, and let $B_0$ be any base hierarchy which does not satisfy the conditions stated in Theorem~\ref{TheoProvKPPin}. Define $B_{i+1} = (B_i)_{+(i+1)}$ for every $i$. Then $\mathsf {KP^-\om} + (\Pi_n)-\mathrm{IND}$ does not prove that for all $n$, there is some $i$ such that $\goodp ni {\mathcal B} = 0$.
\end{theorem}

\section{Concluding remarks}\label{secConc}

We have determined the precise ordinal interpretation for the fractal Goodstein process and used this to establish various independence results for theories between $\sf RCA_0$ and $\sf KP$. In particular, when working with the ouroboros successors, we see that the structure of the first base hierarchy in our dynamical system decides the proof-theoretic strength of termination of the fractal Goodstein process. Having the bases in the first hierarchy farther apart generally results in longer termination times.

Besides letting the first base hierarchy vary, we can obtain independence results for intermediate theories by using a different type of successor than the ouroboros one. For this a different ordinal assignment would be needed that is more suitable for the specific successor at hand. 
In a sense, the ouroboros successors are even more dense than the greedy successors which are used in \cite{fernandez2025fractal}. A first step would be to study the ordinal assignments for sparser successors, like the minimalistic one from Example~\ref{exSucc}.

In \cite{fernandez2025fractal}, the question of combining the fractal Goodstein process with the Ackermann function is also considered. We believe that this approach leads to a principle with even higher proof-theoretic strength.



\bibliographystyle{plain}
\bibliography{bibliogr}

\affiliationone{Department of Mathematics WE16\\
Ghent University\\
Ghent,Belgium
\email{\tt  Milan.Morreel@UGent.be\\Andreas.Weiermann@UGent.be}\\
}
\affiliationtwo{Department of Philosophy\\
University of Barcelona\\
Barcelona, Spain\\
\email{\tt fernandez-duque@ub.edu}
}

\end{document}